\documentclass[11pt]{amsart}

\usepackage{epsfig,amsmath,amsfonts,latexsym,mathrsfs,amssymb,flafter}
\usepackage{color}
\usepackage{overpic}
 \usepackage{palatino}
 \usepackage{hyperref}
 \usepackage[nocompress]{cite}
 \usepackage{tikz,tikz-cd}
\usetikzlibrary{calc,decorations.pathmorphing,shapes,cd}

\newcounter{sarrow}

\newtheorem{theorem}{Theorem}[section]
\newtheorem{proposition}[theorem]{Proposition}
\newtheorem{lemma}[theorem]{Lemma}
\newtheorem{corollary}[theorem]{Corollary}
\newtheorem{conjecture}[theorem]{Conjecture}
\newtheorem{question}[theorem]{Question}

\theoremstyle{definition}
\newtheorem{definition}[theorem]{Definition}
\newtheorem{remark}[theorem]{Remark}

\theoremstyle{remark}

\newcommand{\s}{\mathfrak{s}}

\newcommand{\TT}{\mathfrak{t}}
\newcommand{\im}{\mbox{Im}}

\newcommand{\spinc}{{\mbox{spin$^c$} }}

\newcommand{\zee}{\mathbb{Z}}
\newcommand{\arr}{\mathbb{R}}
\newcommand{\cue}{\mathbb{Q}}
\newcommand{\cee}{\mathbb{C}}
\newcommand{\F}{\mathbb{F}}
\newcommand{\eff}{\mathcal{F}}

\newcommand{\A}{{\mathcal A}}

\renewcommand{\S}{{\mathcal S}}

\renewcommand{\L}{{\mathcal L}}

\newcommand{\cc}{{\mbox{\bf c}}}
\newcommand{\ts}{\textstyle}

\newcommand{\hfhat}{\widehat{HF}}
\newcommand{\cfhat}{\widehat{CF}}

\newcommand{\cfkhat}{\widehat{CFK}}
\DeclareMathOperator{\tb}{tb}
\DeclareMathOperator{\rot}{rot}
\DeclareMathOperator{\tw}{tw}
\DeclareMathOperator{\Int}{Int}

\newcommand{\hmbar}{{\overline{HM}}}
\newcommand{\hmfrom}{\widehat{HM}}
\DeclareFontFamily{U}{mathx}{\hyphenchar\font45}
\DeclareFontShape{U}{mathx}{m}{n}{
      <5> <6> <7> <8> <9> <10>
      <10.95> <12> <14.4> <17.28> <20.74> <24.88>
      mathx10
      }{}
\DeclareSymbolFont{mathx}{U}{mathx}{m}{n}
\DeclareFontSubstitution{U}{mathx}{m}{n}
\DeclareMathAccent{\widecheck}{0}{mathx}{"71}
\newcommand{\hmto}{\widecheck{HM}}

\theoremstyle{plain}

\begin{document}

\title{On contact type hypersurfaces in 4-space}

\author{Thomas E. Mark}

\author{B\"{u}lent Tosun}

\address{Department of Mathematics \\ University of Virginia \\ Charlottesville \\ VA}

\email{tmark@virginia.edu}

\address{Department of Mathematics\\ University of Alabama\\Tuscaloosa\\AL}

\email{btosun@ua.edu}


\begin{abstract} We consider constraints on the topology of closed 3-manifolds that can arise as hypersurfaces of contact type in standard symplectic $\arr^4$. Using an obstruction derived from Heegaard Floer homology we prove that no Brieskorn homology sphere admits a contact type embedding in $\arr^4$, a result that has bearing on conjectures of Gompf and Koll\'ar. This implies in particular that no rationally convex domain in $\cee^2$ has boundary diffeomorphic to a Brieskorn sphere. We also give infinitely many examples of contact 3-manifolds that bound Stein domains in $\cee^2$ but not domains that are symplectically convex with respect to the standard symplectic structure; in particular we find Stein domains in $\cee^2$ that cannot be made Weinstein with respect to the ambient symplectic structure while preserving the contact structure on their boundaries. Finally, we observe that any strictly pseudoconvex, polynomially convex domain in $\cee^2$ having rational homology sphere boundary is diffeomorphic to the standard 4-ball.
\end{abstract}

\maketitle

\section{Introduction}

Let $(X,\omega)$ be a symplectic manifold. A smooth submanifold $Y\subset X$  is said to be a {\it hypersurface of contact type} if it is of codimension 1, and there exists a vector field $v$ defined in a neighborhood of $Y$ and transverse to $Y$, which is Liouville in the sense that $\L_v\omega = \omega$, where $\L_v$ denotes the Lie derivative. In this situation the 1-form $\alpha = \iota_v\omega|_Y$ is a contact form on $Y$, inducing a contact structure $\xi = \ker\alpha\subset TY$. 
Contact type hypersurfaces were introduced by Weinstein \cite{weinstein78}, who conjectured that the characteristic (Reeb) vector field on a compact contact type hypersurface must always admit a closed orbit. This conjecture was proved for hypersurfaces in standard $\arr^{2m}$ by Viterbo \cite{viterbo87}, but Weinstein's conjecture and generalized versions thereof sparked a long series of new ideas in symplectic and contact geometry that are still under exploration.

We are concerned with the question of {\it which smooth, closed, oriented manifolds can be realized}, up to diffeomorphism, as contact type hypersurfaces in $X = \arr^{2m}$ with the standard symplectic structure. When $m>2$, there are many diffeomorphism types of  hypersurfaces arising this way. Indeed, work of Cieliebak-Eliashberg \cite{CE-qconvex} implies that if $W\subset \arr^{2m}$ is a smooth, compact, codimension-0 submanifold that admits a defining Morse function having no critical points of index greater than $m$, then $W$ is isotopic to a Weinstein domain symplectically embedded in $(\arr^{2m},\omega_{std})$, and therefore in this case the boundary $\partial W$ is a hypersurface of contact type. In other words, when $W$ admits a smooth embedding in $\arr^{2m}$, and satisfies some basic topological constraints necessary to admit a Weinstein structure, it can be realized as such inside $\arr^{2m}$. When $m = 2$, which is our focus in this paper, the situation is rather different. The first examples of this difference are due to Nemirovski-Siegel \cite{NemSie}, whose results are discussed further below. Here we introduce an obstruction to contact type embeddings of 3-manifolds derived from Floer homology, which applies in much greater generality than Nemirovski-Siegel's argument (though curiously it does not apply to the examples in \cite{NemSie}). We use this obstruction to further demonstrate that, in contrast to the higher-dimensional situation, there are subtle obstructions to contact type embeddings in $\arr^4$, including cases of contact integer homology 3-spheres that can embed as the boundary of contractible Stein domains, but not as contact type hypersurfaces. In particular, while such Stein domains admit Weinstein structures, the Weinstein structure cannot be embedded symplectically in $(\arr^4,\omega_{std})$. 

The main result of this paper is the following, which rules out a large class of 3-dimensional homology spheres from arising as contact type hypersurfaces. 

\begin{theorem}\label{mainthm} Let $Y =\Sigma(a_1,\ldots, a_n)$ be a 3-dimensional Brieskorn integer homology sphere, oriented as the link of a Brieskorn complete intersection singularity. Then there is no orientation-preserving diffeomorphism between $Y$ and any hypersurface of contact type in $(\arr^4, \omega_{std})$.
\end{theorem}

Recall that a 3-dimensional Brieskorn sphere $Y = \Sigma(a_1,\ldots, a_n)$ is the link of a certain complex surface singularity, and the diffeomorphism type of $Y$ is determined uniquely by a collection of integers $(a_1,\ldots, a_n)$, where $n\geq 3$ and $a_j\geq 2$ (see Section \ref{estimatesec}, or \cite{savelievbook} for example). The manifold $\Sigma(a_1,\ldots, a_n)$ is an integer homology sphere if and only if $a_1,\ldots, a_n$ are pairwise relatively prime \cite[Theorem 4.1]{neumannraymond}. We will use the term ``integer homology sphere'' or just ``homology sphere'' to mean a 3-manifold with the integral homology of $S^3$, and ``rational homology sphere'' for the same condition with rational coefficients.

To put Theorem \ref{mainthm} in a broader context, consider the following sequence of increasingly strong conditions on a closed, oriented 3-manifold $Y$ (an integer homology sphere) concerning 4-manifolds that may have $Y$ as their boundary.
\begin{itemize}
\item[B1.] There is a smooth 4-manifold $W$ having $\widetilde{H}_*(W;\zee) = 0$, with $\partial W \cong Y$.
\item[B2.] There is a smooth embedding $Y \to \arr^4$.
\item[B3.] There is an embedding of $Y$ into $(\arr^4, \omega_{std})$ as a hypersurface of contact type.
\end{itemize}
Certainly B3$\implies$B2$\implies$B1, where for the second implication one can simply take $W$ to be the closure of the bounded component of $\arr^4 - Y$.

The question of which 3-manifolds satisfy B1 is of great interest in smooth low-dimensional topology, and is far from settled even among Brieskorn spheres: many Brieskorn spheres satisfy B1 (and also B2), many do not, and for many the answer is unknown. For example, several infinite families of Brieskorn homology spheres arise as the boundaries of smooth, contractible 4-manifolds that admit Morse functions with a single critical point of each index $0, 1$ and $2$, as constructed by Casson and Harer \cite{CH}. Any 4-manifold with these properties admits a smooth embedding in $\arr^4$ (see \cite{mazur}, or \cite[Example 3.2]{gompfemb}). Thus such Brieskorn spheres satisfy both B1 and B2. On the other hand, a great many Brieskorn spheres are known {\it not} to satisfy B1 or B2 \cite{FS:seifert,furuta90}; the general classification of Brieskorn spheres satisfying these properties is unknown. Under the additional geometric condition of a contact type embedding, however, we find that the problem becomes tractable and Theorem \ref{mainthm} gives a uniform answer. Moreover, in light of Casson-Harer's constructions and the higher-dimensional results of Cieliebak-Eliashberg mentioned above, the theorem  illustrates the contrast between dimension 4 and higher dimensions.

The preceding conditions concern only the smooth and symplectic features of $\arr^4 = \cee^2$, but one can also study embedding questions from the point of view of complex geometry. Of particular interest are Stein domains, by which we mean compact domains\footnote{Following \cite{CE-qconvex}, by a {\it compact domain} we mean a compact set that is the closure of a connected open set.} $W\subset \cee^2$ described as a sub-level set $\{\phi\leq c\}$, where $\phi$ is a proper, strictly plurisubharmonic function defined on a neighborhood of $W$ with regular value $c$. By definition, a choice of such $\phi$ determines a K\"ahler form $\omega_\phi = -dd^c\phi$ on $W$, and the gradient of $\phi$ is a Liouville field for $\omega_\phi$ inducing a contact structure on the boundary $Y$. Work of Gompf \cite{gompfemb} provides many examples of contact 3-manifolds that embed in $\cee^2$ as the boundaries of Stein domains, including many homology spheres. 

However, there is no need for $\omega_\phi$ to agree with the standard symplectic form on $\arr^4$, and in particular the boundary of a Stein domain need {\it not} be a contact type hypersurface in $(\arr^4, \omega_{std})$. To put this another way, in general the form $\omega_\phi$ need not extend to a symplectic form on $\arr^4$: this extension property characterizes a {\it rationally convex} domain $W$. Recall that a compact domain $W\subset \cee^2$ is rationally convex if for every point $x\in \cee^2 - W$, there exists a rational function $r$ such that $|r(x)|> \sup \{|r(w)|, \, w\in W\}$. By a result of Duval and Sibony \cite{DS}, (see also Nemirovski \cite{Nem}), a Stein domain $W$ is rationally convex if and only if it admits a strictly plurisubharmonic function $\phi$ such that $\omega_\phi$ extends to a K\"ahler form on all of $\cee^2$. In this case, one may additionally assume that the extended form is standard outside of a compact subset (cf. \cite[Lemma 3.4]{CE-qconvex}). By a classical result of Gromov \cite{gromov85}, a symplectic form on $\arr^4$ that is standard at infinity in this sense is symplectomorphic to the standard one. Thus from this point of view, one can think of rational convexity (loosely) as the intersection between the Stein condition and the condition of contact type boundary in $(\arr^4, \omega_{std})$, up to ambient symplectomorphism. We note that our terminology departs slightly from standard usage in several complex variables. Our ``Stein domains'' correspond to the closures of strictly pseudoconvex domains, while for us a ``rationally convex domain'' means the closure of a strictly pseudoconvex domain, which is also rationally convex.

With these considerations in mind we can add two additional conditions to our previous list:
\begin{itemize}
\item[B3'.] There is an embedding of $Y$ into $\cee^2$ as the boundary of a Stein domain.
\item[B4.] There is an embedding of $Y$ into $\cee^2$ as the boundary of a rationally convex domain.
\end{itemize}
Thus we have the chain of implications:
\begin{equation}\label{implications}
\begin{tikzcd}
{}& 
\underset{\mbox{\scriptsize (cont. type)}}{\mbox{B3}} \arrow[dr,Rightarrow] &{} & {}\\
 \underset{\mbox{\scriptsize (rat. conv.)}}{\mbox{B4}}\arrow[ur,Rightarrow]\arrow[dr,Rightarrow]&{}& \underset{\mbox{\scriptsize (smooth emb.)}}{\mbox{B2}}\arrow[r,Rightarrow,"Y=\zee HS^3"]& \underset{\mbox{\scriptsize (homol. ball)}}{\mbox{B1}}\\
{} &  \underset{\mbox{\scriptsize (Stein)}}{\mbox{B3'}} \arrow[ur,Rightarrow] & {} & {} 
\end{tikzcd}
\end{equation}
An obvious consequence of Theorem \ref{mainthm} is therefore:

\begin{corollary} No Brieskorn homology sphere is orientation-preserving diffeomorphic to the boundary of a rationally convex domain in $\cee^2$.
\end{corollary}

On the other hand, our results leave open the following tantalizing conjecture of Gompf \cite[Conjecture 3.3]{gompfemb}:

\begin{conjecture}[Gompf]\label{gompfconj}  No nontrivial Brieskorn sphere, with either orientation, arises as the boundary of a Stein domain in $\cee^2$.
\end{conjecture}

 While Conjecture \ref{gompfconj} concerns Brieskorn spheres with either orientation, Theorem \ref{mainthm} applies only to the standard orientation. Note that a symplectic manifold is canonically oriented, and if $Y$ is a contact type hypersurface then a choice of Liouville vector field $v$  induces an orientation on $Y$. On the other hand, if $Y$ is a closed, compact hypersurface in $\arr^{2m}$, then $Y$ is the boundary of a compact domain $W\subset \arr^{2m}$ and hence inherits an orientation from $W$. We will generally assume that these orientations agree, which is the same as saying that the Liouville field near $Y = \partial W$ is directed out of $W$, and always holds if $H^1(Y;\arr) = 0$ (this follows from a simple argument using Stokes' theorem, similar to \cite[Lemma 2.1]{weimin}). A symplectic manifold with an outward-pointing Liouville vector field near its boundary is called {\it symplectically convex}; with this terminology a hypersurface of contact type in $(\arr^{2m}, \omega_{std})$ having vanishing first Betti number is the same as the boundary of a symplectically convex domain.

Besides Brieskorn spheres that do not satisfy B2, or those for which the classification of contact structures has been obtained and shows that no fillable structure has the correct homotopy class to be filled by a homology ball (such as the family $-\Sigma(2,3,6n+1)$ found in \cite{MT:pseudoconvex}), the only direct evidence for Conjecture \ref{gompfconj} was given in \cite{MT:pseudoconvex}. There it was shown that a certain contractible domain in $\cee^2$, having boundary the Brieskorn manifold $\Sigma(2,3,13)$, is not diffeomorphic to a Stein domain. Yet this does not show $\Sigma(2,3,13)$ does not satisfy B3', since it is still conceivable that a different embedding of $\Sigma(2,3,13)$ in $\cee^2$ bounds a Stein domain not diffeomorphic to the one considered in \cite{MT:pseudoconvex}.

Remarkably, the only Brieskorn spheres $\Sigma(a_1,\ldots, a_n)$ known to satisfy B1 have $n = 3$. In fact, the following is a longstanding conjecture appearing in a paper of Koll\'ar, who explains its relationship to the Montgomery-Yang problem concerning the classification of circle actions on $S^5$, as well as the Bogomolov-Miyaoka-Yau inequality.

\begin{conjecture}(see \cite[Conjecture 20]{kollar}, also \cite{FS87}) \label{kollarconj} No Brieskorn sphere $\Sigma(a_1,\ldots, a_n)$ with $n\geq 4$ is the boundary of a smooth integer homology ball.
\end{conjecture}

One could consider a weaker version of this conjecture, in which the condition that $\Sigma(a_1,\ldots, a_n)$ bounds a homology ball is replaced by the condition that it embeds in $\arr^4$. Our results then confirm a symplectic version of this weaker conjecture, ruling out an embedding as a contact type hypersurface (and this version holds for the case $n= 3$ as well). 

It is natural to ask whether any of the implications in the diagram \eqref{implications} can be reversed, or whether any hold between B3 and B3'. In this direction, recall that in \cite{NemSie}, Nemirovski and Siegel classified those disk bundles over compact surfaces that admit rationally convex embeddings in $\cee^2$. In particular they found two cases (certain disk bundles over $\arr P^2$ and the Klein bottle) that embed in $\cee^2$ as Stein domains, but not as rationally convex domains. In fact, the arguments of \cite{NemSie} imply that if $M$ is the boundary of the disk bundle over $\arr P^2$ having Euler number $-2$, then $M$ satisfies B3', but not B3. Each of the conditions B3, B3', B4 can be considered with respect to a particular contact structure on $Y$, and if the contact structure is fixed, it is not difficult to give many more examples of manifolds satisfying B3' but not B3. We give an infinite family of such examples in Section \ref{examplesec} (Nemirovski-Siegel's argument is also specific to a particular contact structure, but in their case there is only one relevant contact structure to consider). Even with a fixed contact structure, it seems an interesting and delicate question whether there exist 3-manifolds satisfying B3 but not B3'; see Section \ref{hypexsec} for additional remarks.

Returning to symplectic topology, Weimin Chen conjectured in \cite[Conjecture 5.1]{weimin} that the only rational homology sphere arising as the boundary of a rationally convex domain in $\cee^2$ is the 3-sphere. In this direction, faced with the current lack of examples, we can ask:

\begin{question}\label{ctquestion} Does any rational homology 3-sphere, other than $S^3$, admit an embedding in $(\arr^4, \omega_{std})$ as a contact type hypersurface?
\end{question}

Chen was motivated by questions in smooth 4-dimensional topology, particularly whether there exist smooth 4-manifolds homeomorphic but not diffeomorphic to the projective plane $\cee P^2$. He shows that such an exotic 4-manifold would result if one were able to find a suitably ``small'' concave filling or ``cap'' of the circle bundle $M_{NS}$ over $\arr P^2$ considered by Nemirovski-Siegel---essentially by attaching such a cap to the Stein domain bounded by $M_{NS}$. The proof that the result would not be diffeomorphic to $\cee P^2$ relies on the fact that $M_{NS}$ cannot be realized as a hypersurface of contact type in $\cee P^2$, which is a generalization of Nemirovski-Siegel's result that $M_{NS}$ does not bound a rationally convex domain in $\cee^2$ (with nearly the same proof). While construction of small concave caps remains an open problem, and our techniques apply to domains in $\cee^2$ rather than $\cee P^2$, one may hope that the examples of Stein domains not having contact type boundary given here may allow some additional flexibility in this effort.

\begin{remark} As a partial response to Chen's conjecture concerning rationally convex domains with rational homology sphere boundary, we point out the following consequence of results in the literature:

\begin{theorem}\label{polyconvthm} If $W\subset \cee^2$ is a polynomially convex domain whose boundary is a rational homology sphere, then $W$ is diffeomorphic to $B^4$ and in particular $\partial W \cong S^3$.
\end{theorem}

Here, the definition of polynomial convexity is identical to that of rational convexity with rational functions replaced by polynomials (with the same caveat regarding terminology as before). To prove the theorem, note that it is a consequence of work of Oka \cite{oka} (see also \cite[Remark 3.3]{CE-qconvex}) that if $W$ is as in the statement then one can add handles of index $\leq 2$ to $W$ in $\cee^2$ to obtain a 4-ball $B^4\subset \cee^2$. Then $B^4- \Int(W)$ is a rational homology cobordism from $\partial W$ to $S^3$, which is ribbon in the sense of \cite{DLVVW}. By \cite[Theorem 1.14]{DLVVW} (which is essentially a result of Gordon \cite{gordon81} combined with Perelman's geometrization theorem) it follows that $\partial W\cong S^3$, and in particular $W$ is a Stein filling of $S^3$. It was proved by Eliashberg \cite{E-filling} that any such filling is diffeomorphic to $B^4$. \hfill$\Box$
\end{remark}

Our obstruction to contact type embeddings can be described as follows. Recall that the Heegaard Floer homology of a compact, oriented 3-dimensional manifold $Y$ is a vector space over the field $\F = \zee/2\zee$, written $HF^+(Y)$. This vector space is always infinite-dimensional over $\F$, but has a natural finite-dimensional quotient $HF^+_{red}(Y)$, the {\it reduced} Heegaard Floer homology. If $\xi$ is a contact structure on $Y$, there is an associated element $c^+(\xi)\in HF^+(-Y)$ that is an invariant of $\xi$ up to contact isotopy, where $-Y$ denotes the 3-manifold $Y$ with the opposite orientation. If $\xi$ is overtwisted then $c^+(\xi) = 0$ while if $\xi$ is strongly fillable then $c^+(\xi)\neq 0$. Theorem \ref{obstrthm} below (particularly Corollary \ref{reducedcor}) shows that if $(Y,\xi)$ is embedded in $(\arr^4, \omega_{std})$ as the boundary of a symplectically convex domain then, although $c^+(\xi)$ is nonzero, its image in $HF^+_{red}(-Y)$ vanishes. Our proof of Theorem \ref{mainthm} amounts to showing that, for any symplectically fillable contact structure on a Brieskorn sphere that bounds a 4-manifold with the integer homology of the 4-ball, this condition is not satisfied.

Unfortunately, the typical Brieskorn sphere admits {\it many} inequivalent fillable contact structures, and the obstruction just described applies to only one at a time. A key difficulty in the proof of Theorem \ref{mainthm} is that there is currently no general classification of contact structures on Brieskorn spheres. We take an approach that allows us to detect when the image of $c^+(\xi)$ in $HF^+_{red}(-Y)$ is nonzero, without particular assumptions on a symplectic filling of $\xi$. Our strategy is to study two numerical invariants associated to knots $K$ embedded in 3-manifolds: one, written $\tau_{sm}$, depends only on the (smooth) topology of $(Y,K)$, while the other, $\tau_\xi$, reflects aspects of the contact geometry. We observe that, from their definitions, if $c^+(\xi)$ vanishes in $HF^+_{red}(-Y)$, then these two invariants satisfy an inequality (Lemma \ref{taulemma}). Using a new estimate on the {\it twisting number} of any contact structure on a Brieskorn sphere (Theorem \ref{twistthm}), we show that this inequality is violated for any contact structure on a Brieskorn sphere that bounds a homology ball.

The paper is organized as follows.  In Section \ref{obstrsec1} we give the details of the obstruction to rational convexity mentioned above, and in Section \ref{obstrsec2} describe the invariants $\tau_{sm}$ and $\tau_\xi$. Section \ref{estimatesec} contains estimates on $\tau_{sm}$ and $\tau_\xi$ for the regular Seifert fiber in a Brieskorn sphere. In Section \ref{twistsec} we obtain a lower bound for the twisting number of any contact structure on a Brieskorn homology sphere, and this combined with results of Sections \ref{obstrsec} and \ref{estimatesec} proves Theorem \ref{mainthm}. In Section \ref{examplesec1} we deduce some additional consequences of our results for two questions that, it turns out, can be addressed by similar methods: namely, whether a contact structure is supported by an open book with pages of genus 0, and the existence a strong symplectic cobordism to the 3-sphere. The remainder of Section \ref{examplesec} provides some further details and examples regarding the relations between the conditions B1, B2, B3, B3' and B4, including illustrations of the use and limitations of our methods. The reader with particular interest in these questions, including examples of Stein domains in $\cee^2$ that cannot be made Weinstein with respect to $\omega_{std}$, may read Sections \ref{examplesec2} through \ref{unobstrsec} first.

\subsection*{Acknowledgements} The authors are grateful to Bob Gompf for his interest in this project and many useful remarks and suggestions. Our thanks to Stefan Nemirovski for helpful comments on an early draft of this paper. We are also grateful to the referees for their careful reading and many suggestions. Part of this research was carried out in July 2019 while the authors were attending a PCMI workshop on ``Quantum Field Theory and Manifold Invariants". The authors are grateful to PCMI for the support and wonderful research environment. The first author was supported in part by a grant from the Simons Foundation (523795, TM).  Part of the research included in this article was carried out during the second author’s research stay in Montreal in Fall 2019. This research visit was supported in part by funding from the Simons Foundation and the Centre de Recherches Mathmatiques, through the Simons-CRM scholar-in-residence program. The second author is grateful to CRM and CIRGET, and in particular to Steve Boyer for their wonderful hospitality. The second author was also supported in part by grants from  NSF DMS-2105525 and the Simons Foundation (636841, BT). Finally, the authors wish to thank Sarah Zampa for pointing out an error in the original version of this paper.

\section{Floer invariants and contact type embeddings}\label{obstrsec}

\subsection{An obstruction to contact type embedding}\label{obstrsec1}  Recall \cite{OS:hf1,OS:hf2} that a closed, connected, oriented 3-manifold $Y$, equipped with a \spinc structure $\s$, gives rise to a chain complex $CF^-(Y,\s)$ which is a finitely generated free complex over the polynomial ring $\F[U]$. Here $\F$ denotes the field with two elements; the theory can be developed with more general coefficients but we will not need this. Note that the Heegaard Floer chain complex $CF^-(Y,\s)$ depends on certain auxiliary choices, notably a Heegaard decomposition of $Y$, but is independent of such choices up to chain homotopy equivalence. In particular, the homology of $CF^-(Y,\s)$ is a topological invariant of the pair $(Y,\s)$ denoted $HF^-(Y,\s)$, and is a module over $\F[U]$. The sense in which this module itself (rather than its isomorphism type) is an invariant of $(Y,\s)$ is explored in \cite{JTZ}. 

Localizing with respect to the variable $U$ gives rise to a short exact sequence
\[
\begin{tikzcd}
0\arrow[r]& CF^-(Y,\s) \arrow[r]& CF^{\infty}(Y,\s)\arrow[r] &CF^+(Y,\s)\arrow[r]& 0
\end{tikzcd}
\]
of complexes over $\F[U]$, where $CF^\infty = CF^- \otimes \F[U,U^{-1}]$ and $CF^+ = CF^\infty / CF^-$. In obvious notation, the associated long exact sequence in homology reads
\begin{equation}\label{les}
\begin{tikzcd}
\cdots\arrow[r]& HF^{\infty}(Y,\s)\arrow[r]& HF^+(Y,\s)\arrow[r, "\delta"]& HF^-(Y,\s)\arrow[r]& HF^\infty(Y,\s)\arrow[r]&\cdots
\end{tikzcd}
\end{equation}
where the connecting homomorphism $\delta$ induces an isomorphism $\delta: HF^+_{red}(Y,\s)\to HF^-_{red}(Y,\s)$ between the {\it reduced groups} defined by
\[
HF^+_{red}= HF^+/ \mathrm{Im}(HF^\infty\to HF^+) \quad\mbox{and}\quad HF^-_{red} = \ker(HF^-\to HF^\infty).
\]

The induced action of $U$ on $CF^+$ is surjective, and we have another short exact sequence
\begin{equation}\label{plusseq}
\begin{tikzcd}
0\arrow[r]& \cfhat(Y,\s)\arrow[r]& CF^+(Y,\s) \arrow[r, "U"]& CF^+(Y,\s)\arrow[r]& 0
\end{tikzcd}
\end{equation}
where $\cfhat = \ker(U)$. The associated sequence in homology is
\begin{equation}\label{hatles}
\begin{tikzcd}
\cdots \arrow[r]& \hfhat(Y,\s)\arrow[r]&HF^+(Y,\s)\arrow[r, "U"]& HF^+(Y,\s)\arrow[r]& \cdots.
\end{tikzcd}
\end{equation}

Heegaard Floer homology is functorial under cobordisms, in the following sense. By a cobordism $W: Y_1\to Y_2$ we mean a smooth, compact, connected, oriented 4-manifold with $\partial W = -Y_1\sqcup Y_2$ as oriented manifolds, where $Y_1$, $Y_2$ are closed oriented 3-manifolds as above and $-Y_1$ means the 3-manifold $Y_1$ with its orientation reversed. In this situation, a choice of \spinc structure $\TT$ on $W$ (together with some auxiliary choices) gives rise to $\F[U]$-chain maps between Heegaard Floer complexes for $Y_1$ and $Y_2$, compatible with the sequences above, whose associated homomorphism on homology depends only on the diffeomorphism type of $(W,\TT)$. Specifically, we have $\F[U]$-homomorphisms
\[
F_{W,\TT}^\circ: HF^\circ(Y_1, \s_1)\to HF^\circ(Y_2, \s_2)
\]
where ${}^\circ$ indicates any of the flavors $-, \infty, +$ and where $\s_i = \TT|_{Y_i}$. By changing the roles of $Y_1$ and $Y_2$, the same manifold $W$ can also be thought of as a cobordism from $-Y_2$ to $-Y_1$, which we write $\overline{W}$. Hence: 
\[
F_{\overline{W},\TT}^\circ: HF^\circ(-Y_2, \s_2)\to HF^\circ(-Y_1, \s_1).
\]
Now suppose $Y$ is equipped with a (co-oriented, positive) contact structure $\xi$. In \cite{OS:contact}, Ozsv\'ath and Szab\'o defined an element $\hat{c}(\xi) \in \hfhat(-Y, \s_\xi)$ that is an invariant of the isotopy class of contact structures determined by $\xi$. Here $\s_\xi$ is the \spinc structure naturally determined by the oriented plane field $\xi\subset TY$; note that \spinc structures on $Y$ and $-Y$ are in natural bijection. This invariant, which we call the {\it contact invariant} associated to $\xi$, has led to remarkable progress in 3-dimensional contact topology, and it plays a key role in our study of symplectic convexity. 

Recall that a {\it strong symplectic cobordism} between contact manifolds $(Y_1,\xi_1)$ and $(Y_2,\xi_2)$ is a cobordism $W: Y_1\to Y_2$ in the sense above, equipped with a symplectic form $\omega$ such that:
\begin{itemize}
\item There is a vector field $v$ defined near $\partial W$ and transverse to $\partial W$, pointing into $W$ along $Y_1$ and out of $W$ along $Y_2$, which is Liouville for $\omega$.
\item The 1-forms $\alpha_i\in \Omega^1(Y_i)$ given by the restrictions of $\iota_v\omega$ to $Y_i$, which are contact forms by the Liouville condition, define the contact structures $\xi_i$.
\end{itemize}

In the following, we write $c^+(\xi)$ for the image of $\hat{c}(\xi)$ under the map $\hfhat(-Y)\to HF^+(-Y)$ \eqref{hatles}, and $(S^3, \xi_0)$ for the 3-sphere with its standard isotopy class of tight contact structure.

\begin{theorem}\label{obstrthm} Let $(Y,\xi)$ be a contact 3-manifold, and $\xi_0$ the standard contact structure on the 3-dimensional sphere. If there is a strong symplectic cobordism from $(Y,\xi)$ to $(S^3, \xi_0)$, then the contact invariant $c^+(\xi)$ lies in the image of the map $HF^\infty(-Y, \s_\xi)\to HF^+(-Y, \s_\xi)$. 
\end{theorem}
\begin{proof} Recall that there is a parallel theory to Heegaard Floer homology, developed by Kronheimer and Mrowka \cite{KMbook} and called monopole Floer homology. Monopole Floer groups come in three ``flavors'' $\hmbar(Y,\s)$, $\hmfrom(Y,\s)$ and $\hmto(Y,\s)$ that are modules over $\F[U]$ related by a long exact sequence analogous to \eqref{les}. Deep work of Kutluhan, Lee and Taubes \cite{KLTall} and (independently) of Colin, Ghiggini and Honda \cite{CGHall}, relying on additional results of Taubes \cite{TaubesECH1}, shows that there are $U$-equivariant isomorphisms
\[
\hmbar(Y,\s)\cong HF^\infty(Y,\s)\qquad \hmfrom(Y,\s)\cong HF^-(Y,\s)\qquad \hmto(Y,\s)\cong HF^+(Y,\s),
\]
and these maps commute with those in the long exact sequences. Moreover, there is an analog of the invariant $c^+(\xi)$ in monopole Floer homology, which we write simply as $\cc(\xi)$, and the isomorphism $\hmto(-Y,\s_\xi)\to HF^+(-Y,\s_\xi)$ sends one contact invariant to the other. (This is stated in \cite[Theorem 8.1]{CGHsummary} where $\hat{c}(\xi)$ is identified with a corresponding invariant in embedded contact homology, and the invariant $c^+(\xi)$ corresponds as well by the results of \cite{CGH3}. That the contact invariant in embedded contact homology corresponds to the monopole contact invariant $\cc(\xi)$ follows from work of Taubes \cite{TaubesECH1}.) In light of these correspondences, it suffices for the theorem to prove the corresponding statement for the contact invariant in monopole Floer homology.

The essential input for that proof is a result of Echeverria \cite{mariano}, building on work of Mrowka and Rollin \cite{MR06}. According to the main result of \cite{mariano}, if $(W,\omega): (Y_1,\xi_1)\to (Y_2,\xi_2)$ is a strong symplectic cobordism, then the homomorphism $\hmto(-Y_2,\s_{\xi_2})\to \hmto(-Y_1,\s_{\xi_1})$ induced by $\overline{W}$, equipped with the \spinc structure associated to $\omega$, carries the contact invariant of $\xi_2$ to that of $\xi_1$. 

The result now follows quickly for algebraic reasons, as observed in Corollary 10 of \cite{mariano}. Let $(W,\omega)$ be a strong cobordism to $(S^3,\xi_0)$ as in the statement. Since the homomorphism $j_*: \hmto(-S^3)\to \hmfrom(-S^3)$ vanishes we have $0 = j_*F_{\overline W}(\cc(\xi_{std})) = j_*(\cc(\xi))$, which is equivalent to $\cc(\xi)$ lying in $\mathrm{Im}(\hmbar(-Y,\s_{\xi})\to \hmto(-Y,\s_\xi))$. Applying the isomorphism between the monopole and Heegaard Floer theories in dimension 3 gives the conclusion.
\end{proof}


We remark that the equivalence between monopole and Heegaard Floer homology is not known to extend to the maps induced by cobordisms. In particular, while it is known that the Heegaard Floer contact invariant behaves naturally under Stein cobordisms, for example, its behavior under a general strong symplectic cobordism has not yet been established.

\begin{corollary}\label{reducedcor} Let $W$ be a compact domain in $\arr^4$, with smooth, connected boundary, and assume there exists a Liouville field for $\omega_{std}$ defined near $\partial W$ and directed transversely out of $\partial W$. Let $(Y,\xi)$ be the boundary of $W$ with the induced contact structure and orientation. Then $c^+(\xi)$ is nonzero and in the image of the homomorphism $HF^\infty(-Y,\s_\xi)\to HF^+(-Y,\s_\xi)$. Equivalently, the image of $c^+(\xi)$ in $HF^+_{red}(-Y)$ is zero.
\end{corollary}

\begin{proof} That $c^+(\xi)\neq 0$ follows since $\xi$ admits a strong symplectic filling \cite[proof of Theorem 2.13]{Ghiggini:fillability}). Choose a sufficiently large ball $B_R^4\subset \arr^4$, so that $W\subset \mathrm{int}(B^4_R)$, and let $Z = B^4_R - \mathrm{int}(W)$. Then $Z$ with the restriction of $\omega_{std}$ is a strong symplectic cobordism from $(Y,\xi)$ to $(S^3_R, \xi_0)$, with the given Liouville field near $\partial W$ together with the standard radial Liouville field near $S^3_R$. The result follows from the previous theorem.
\end{proof}

As in the introduction, we call a domain $W\subset (\arr^4, \omega_{std})$ with outwardly-oriented contact type boundary as in the Corollary a {\it symplectically convex} domain.

We note that if $Y$ is smoothly embedded in $\arr^4$, then $Y$ is the boundary of a smooth, compact domain $W\subset \arr^4$. If the embedding is of contact type, and the Liouville field $v$ near $Y$ extends across $W$ as a Liouville field for $\omega_{std}$, then an easy argument using Stokes' theorem implies that $v$ is directed out of $\partial W$. In other words, $W$ is symplectically convex and the orientation on $Y$ induced by $v$ agrees with that induced by $W$. It is not hard to check that this condition holds whenever $Y$ is a rational homology sphere.


\begin{remark}
The condition of symplectic convexity is essential in the above corollary; in particular it is not equivalent to assuming that $W$ is (strictly) pseudoconvex, i.e. Stein. Indeed, there are many examples of contractible Stein manifolds that embed as Stein domains in $\cee^2$, for which the contact invariant of the induced contact structure on the boundary does not satisfy the conclusion of Corollary \ref{reducedcor} (see Section \ref{examplesec} below). Note that if $W$ is pseudoconvex then in the notation of the proof above, the manifold $Z = B^4_R - \mathrm{int}(W)$, with the restriction of the standard K\"ahler form from $\cee^2$, is a {\it weak} symplectic cobordism from $(Y, \xi)$ to $(S^3, \xi_0)$ and is strong at $S^3$. In particular the hypotheses of Echeverria's naturality result cannot in general be relaxed to allow such weak cobordisms.
\end{remark}

\begin{remark} If, in the situation of Corollary \ref{reducedcor}, the manifold $Y$ is a rational homology sphere, then it is not hard to show that the Liouville fields defined near $\partial Z$ extend to a Liouville field on all of $Z$. In other words, $Z$ is a {\it Liouville cobordism} from $(Y, \xi)$ to $(S^3,\xi_0)$. If this Liouville structure can always be strengthened to a Weinstein structure, then as in Theorem \ref{polyconvthm}, work of Gordon \cite{gordon81} as recast by Daemi--Lidman--Vela-Vick--Wong \cite{DLVVW} implies that in fact $Y$ must always be the 3-dimensional sphere: this would provide a negative answer to Question \ref{ctquestion}. There exist examples of Liouville cobordisms that cannot be made Weinstein, however (see \cite{mcduff91,geiges95} for example), and in general the distinction between the two is subtle. Indeed, if we suppose that $W$ is pseudoconvex, then the hypothesis in Corollary \ref{reducedcor} is the same as asking that $W$ be isotopic to a rationally convex domain. The distinction between a Liouville and Weinstein structure on $Z$ is then equivalent to the distinction between rational convexity and polynomial convexity of the domain $W$.
\end{remark}

\begin{remark} While our proof of Theorem \ref{mainthm} makes use of the obstruction of Theorem \ref{obstrthm} and Corollary \ref{reducedcor}, it is worth pointing out that we do {\it not} prove that every fillable contact structure $\xi$ on an arbitrary Brieskorn sphere has the property that the image of $c^+(\xi)$ in $HF^+_{red}$ is nonzero. Indeed, this condition is not satisfied by the unique tight contact structure on the Poincar\'e sphere $\Sigma(2,3,5)$, which is Stein fillable.  In that sense our arguments are specific to the embedding problem we consider, but see Corollary \ref{strongcobordcor} below.
\end{remark}

Finally, we point out that the condition obtained in Theorem \ref{obstrthm} has appeared in another context: it was shown in \cite{OSS:planar} that this condition is necessary for a contact structure to be supported by an open book decomposition having pages of genus zero. See Theorem \ref{planarthm} for our results on this question.

\subsection{Floer invariants for knots}\label{obstrsec2}

Our aim is to show that for certain classes of 3-manifolds $Y$ the conclusion of Corollary \ref{reducedcor} does not hold for any fillable contact structure $\xi$, ruling out the existence of a symplectically convex domain with boundary $Y$. To facilitate the discussion, recall that for any \spinc rational homology sphere $(Y,\s)$, there is an isomorphism $HF^\infty(Y,\s)\cong \F[U,U^{-1}]$ \cite[Theorem 10.1]{OS:hf2}, and in fact we have 
\[
 {\mathrm{Im}}(HF^\infty(-Y,\s_\xi)\to HF^+(-Y,\s_\xi)) \cong \F[U,U^{-1}] / \F[U]
 \]
 as $\F[U]$-modules. Since $c^+(\xi)$ is the image of $\hat{c}(\xi)$ it is clear from \eqref{hatles} that for any contact structure $\xi$ on $Y$ the contact invariant satisfies $Uc^+(\xi) = 0$. Let $\Theta^+\in HF^+(-Y, \s_\xi)$ be the unique nonzero element in $\ker(U) \cap {\mathrm{Im}}(HF^\infty(-Y,\s_\xi)\to HF^+(-Y,\s_\xi))$. Then if $c^+(\xi)$ is nonzero and different from $\Theta^+$, it follows from the preceding and Corollary \ref{reducedcor} that $\xi$ does not arise as the induced contact structure at the boundary of a symplectically convex domain in $\arr^4$. Hence, our aim will be to show that for any (relevant, e.g. fillable) contact structure $\xi$ on $Y$, the contact invariant is different from $\Theta^+$. 
 
%
%

Let $K\subset Y$ be a knot, which is to say a smoothly embedded circle. To simplify the statements to follow we assume that $Y$ is an integral homology sphere; in particular there is a unique \spinc structure on $Y$. In this situation Ozsv\'ath and Szab\'o \cite{OS:hfk} show how to associate to $(Y,K)$ a filtration on (a representative of the chain homotopy class of) $CF^-(Y)$; this chain complex with the data of the filtration is written $CFK^-(Y,K)$. It is a (free, finitely generated) module over $\F[U]$ as before, and its filtered chain homotopy type is an invariant of $(Y,K)$.

Localizing with respect to $U$ as before produces the variants $CFK^{\infty}(Y,K)$ and $CFK^+(Y,K)$, and the latter gives rise to the complex $\cfkhat(Y,K)$ as in \eqref{plusseq}. In particular, $\cfkhat(Y,K)$ is a complex homotopic to $\cfhat(Y)$, together with a filtration
\begin{equation}\label{filt}
\cdots \subset \eff_{m-1}\subset \eff_m \subset \eff_{m+1}\subset \cdots \subset \cfhat(Y)
\end{equation}
by subcomplexes $\eff_m$, such that $\eff_m = 0$ for $m \ll 0$ and $\eff_m = \cfhat(Y)$ for $m\gg 0$. 

If we fix a nonzero class $\hat x\in \hfhat(Y)$, then one obtains an integer-valued invariant $\hat{\tau}_{\hat x}(Y,K)$ for knots in $Y$ by declaring
\begin{equation}\label{tauhatdef}
\hat{\tau}_{\hat x}(Y,K) = \min \{m \,|\, {\hat x}\in \im(i_m: H_*(\eff_m)\to \hfhat(Y))\}
\end{equation}
where $i_m$ indicates the map in homology induced by the inclusion $\eff_m \subset \cfhat(Y)$.

Similarly, if one fixes a class $x^+\in HF^+(Y)$ we can define
\begin{equation}\label{tauplusdef}
\tau^+_{x^+}(K) = \min\{ m\,|\, x^+\in\im(\rho\circ i_m: H_*(\eff_m)\to HF^+(Y))\}
\end{equation}
where 
\[
\rho: \hfhat(Y) \to HF^+(Y)
\]
is the homomorphism in the long exact sequence \eqref{hatles}. The notation $\hat{\tau}$ (resp. $\tau^+$) is meant to suggest that the invariant is derived by considering the interaction between the knot filtration and a fixed class in $\hfhat$ (resp. $HF^+$). Note that whenever $\rho(\hat{x})$ is nonzero, we have 
\begin{equation}\label{tauineq}
\tau^+_{\rho(\hat{x})}(Y,K) \leq \hat{\tau}_{\hat{x}}(Y,K)
\end{equation}
for all $K$, with equality if $\rho$ happens to be injective.

The invariant $\tau(K)$ for a knot in $S^3$ is defined (by Ozsv\'ath-Szab\'o \cite{OS:hfkgenus}; a similar construction was considered by Rasmussen \cite{Rasmussen}) to be $\tau(K) = \hat{\tau}_{\hat \Theta}(S^3,K)$ where $\hat \Theta$ is the unique nonzero element of $\hfhat(S^3)$. More generally we have:

\begin{definition}[Raoux \cite{raoux}, Definition 2.2] For a knot $K$ in an integer homology sphere $Y$ the {\it smooth tau-invariant} of $K$ is defined to be
\[
\tau_{sm}(Y,K) = \tau^+_{\Theta^+}(Y,K),
\]
where $\Theta^+$ is the unique nonzero element of $HF^+(Y)$ with $U\Theta^+ = 0$ and $\Theta^+\in \im (HF^\infty(Y)\to HF^+(Y))$.
\end{definition}

Note that for the case $Y = S^3$, the map $\hfhat(S^3)\to HF^+(S^3)$ is injective and carries $\hat{\Theta}$ to $\Theta^+$, and therefore $\tau_{sm}(S^3, Y) = \tau(K)$. We also remark that Raoux's definition is more general than the one we have given, in that it applies to knots in rational homology spheres, where moreover the knot may not be nullhomologous. In this case the filtration \eqref{filt} is more subtle to define, and in particular filtration values may not be integral. Raoux addresses this issue by adjusting filtration values by a number $k_{\s}$ depending on the chosen \spinc structure $\s$ on $Y$; when $Y$ is an integer homology sphere as in our situation, there is a unique \spinc structure $\s$ and $k_\s = 0$. 

Now observe that a knot $K\subset Y$ can be regarded as a knot in $-Y$ by reversing the ambient orientation. It therefore gives a filtration of $\cfhat(-Y)$, and a nonzero class in $\hfhat(-Y)$ gives rise to a $\tau$-invariant as above.

\begin{definition}[Hedden \cite{heddentau}] Let $(Y,\xi)$ be a contact 3-manifold with the property that the contact invariant $\hat{c}(\xi) \in \hfhat(-Y)$ is nonzero. For a knot $K\subset Y$, the {\it contact tau-invariant} of $K$ is defined to be
\[
\tau_\xi(Y,K) = -\hat{\tau}_{\hat{c}(\xi)}(-Y,K).
\]
\end{definition}

This is not quite the definition Hedden gives, but by \cite[Proposition 28]{heddentau}, it is equivalent.

\begin{lemma}\label{taulemma} Let $Y$ be an integer homology sphere and $\xi$ a contact structure on $Y$ such that $c^+(\xi) = \Theta^+\in HF^+(-Y)$. Then for any knot $K\subset Y$, we have
\[
\tau_\xi(Y,K) \leq \tau_{sm}(Y,K).
\]
\end{lemma}

\begin{proof} From \eqref{tauineq}, we see that if $\hat{x}\in \hfhat(Y)$ has $\rho(\hat{x}) = \Theta^+$, then $\tau_{sm}(Y,K) = \tau^+_{\Theta+}(Y,K) \leq \hat{\tau}_{\hat{x}}(Y, K)$. According to \cite[Proposition 3.10(1)]{raoux}, $\tau_{sm}(-Y,K) = -\tau_{sm}(Y,K)$, and therefore since $\Theta^+ = c^+(\xi) = \rho(\hat{c}(\xi))$,
\[
\tau_\xi(Y,K) = -\hat{\tau}_{\hat{c}(\xi)}(-Y,K) \leq -\tau^+_{\Theta+}(-Y,K) = \tau_{sm}(Y,K).
\]
\end{proof}

By Corollary \ref{reducedcor}, the preceding lemma implies:

\begin{corollary}\label{convexcor} If the integer homology sphere $(Y,\xi)$ is the contact boundary of a symplectically convex domain in $\arr^4$, then for any knot $K\subset Y$ we have
\[
\tau_\xi(Y,K) \leq \tau_{sm}(Y,K).
\]
The same is true if, more generally, there exists a strong symplectic cobordism from $(Y,\xi)$ to $(S^3, \xi_0)$.
\end{corollary}

\begin{remark} 
Heuristically, the smooth tau-invariant is defined as the first filtration level $j$ at which $\Theta^+$ is in the image of the homology of $\eff_j$, while the contact tau-invariant is the first $j$ for which the contact invariant $c(\xi)$ is in that image. If $c^+(\xi) = \Theta^+$, one expects these invariants to agree---and therefore, exhibiting a knot for which the two tau-invariants {\it disagree} suffices to show $c^+(\xi)\neq \Theta^+$ and hence $\xi$ is not filled by a symplectically convex domain. Lemma \ref{taulemma} gives only an inequality because strictly Hedden's contact tau-invariant is defined using $\hat{c}(\xi)$ rather than its image $c^+(\xi)$ (and there is the attendant headache of the orientation reversal). This inequality suffices for our purposes, so we work with Hedden's definition. However, there is a modification of Hedden's invariant that {\it is} equal to $\tau_{sm}$ in the situation of Lemma \ref{taulemma}. Namely, for a knot $K$ in a contact manifold $(Y,\xi)$ with the property that $c^+(\xi)$ is nonzero, we can define
\[
\tau^+_\xi(Y,K) = -\tau^+_{c^+(\xi)}(-Y,K)
\]
in the notation of \eqref{tauplusdef}. It is then easy to check that when $c^+(\xi)$ is nonzero,
\begin{enumerate}
\item For any $K$ in $Y$, we have $\tau_\xi(Y,K) \leq \tau^+_\xi(Y,K)$.
\item If $c^+(\xi)$ is in the kernel of the map $HF^+(-Y)\to HF^+_{red}(-Y)$, or equivalently if $c^+(\xi) = \Theta^+$, then $\tau^+_\xi(Y,K) = \tau_{sm}(Y,K)$. 
\end{enumerate}
In particular $\tau^+_\xi(Y,K)$ gives an upper bound for $\tb(K) + |\rot(K)|$ for any Legendrian $K$, since $\tau_\xi$ does (see below), which is the essential property used in our proof. 
\end{remark}

\section{Estimates for Seifert manifolds}\label{estimatesec}

We now turn to Brieskorn spheres to begin the proof of Theorem \ref{mainthm}. By definition, the Brieskorn sphere $\Sigma(a_1,\ldots, a_n)$ is the link of a complete intersection singularity described as follows. Choose a collection of complex constants $\{c_{i,j}\}$ where $1\leq i \leq n$ and $1\leq j \leq n-2$. Then consider the algebraic surface $V(a_1,\ldots, a_n)\subset \cee^n$ given by
\[
V(a_1,\ldots, a_n) = \{ (z_1,\ldots z_n) \, | \, c_{1,j}z_1^{a_1} + \cdots + c_{n,j}z_n^{a_n} = 0, \, j = 1,\ldots, n-2\},
\]
the intersection of $n-2$ hypersurfaces. For sufficiently generic choice of coefficients $c_{i,j}$, this variety has an isolated singularity at the origin, and $\Sigma(a_1,\ldots, a_n)$ is intersection of $V(a_1,\ldots, a_n)$ with a sufficiently small sphere. The diffeomorphism type of $\Sigma(a_1,\ldots, a_n)$ depends only on $(a_1,\ldots, a_n)$. From now on we suppose that $n\geq 3$ and each $a_j\geq 2$, and moreover that $a_1,\ldots, a_n$ are pairwise relatively prime. The latter condition is equivalent to the assumption that $\Sigma(a_1,\ldots, a_n)$ is an integer homology sphere. See \cite{neumann77,neumannraymond,milnor:brieskorn,savelievbook} for additional details and references.

As the link of a weighted homogeneous singularity $\Sigma(a_1,\ldots, a_n)$ carries an action of the circle, leading to a description as a Seifert manifold and a surgery presentation, which are obtained as follows. We can find integers $(b, b_1, \ldots, b_n)$ such that
\begin{equation}\label{seifeqn}
a_1\cdots a_n \sum_k \frac{b_k}{a_k} = 1 + b\cdot a_1\cdots a_n,
\end{equation}
and in fact the numbers $b$, $b_1,\ldots, b_n$ are determined by \eqref{seifeqn} up to the simultaneous replacement of any $b_j$ by $b_j \pm a_j$ and $b$ by $b\pm 1$. (Thus $b_j$ is uniquely determined modulo $a_j$.) In particular, we can arrange that $b = 0$, in which case the collection $(a_1,b_1),\ldots, (a_n, b_n)$ are said to be {\it unnormalized} Seifert invariants for $\Sigma(a_1,\ldots, a_n)$. 
Then $\Sigma(a_1,\ldots, a_n)$ is diffeomorphic to the 3-manifold specified by the surgery diagram in Figure \ref{small}. (We are following the notation and conventions of \cite{savelievbook}.)
\begin{figure}[h!]
\begin{center}
\def\svgwidth{1.5in}
\begingroup%
  \makeatletter%
  \providecommand\color[2][]{%
    \errmessage{(Inkscape) Color is used for the text in Inkscape, but the package 'color.sty' is not loaded}%
    \renewcommand\color[2][]{}%
  }%
  \providecommand\transparent[1]{%
    \errmessage{(Inkscape) Transparency is used (non-zero) for the text in Inkscape, but the package 'transparent.sty' is not loaded}%
    \renewcommand\transparent[1]{}%
  }%
  \providecommand\rotatebox[2]{#2}%
  \newcommand*\fsize{\dimexpr\f@size pt\relax}%
  \newcommand*\lineheight[1]{\fontsize{\fsize}{#1\fsize}\selectfont}%
  \ifx\svgwidth\undefined%
    \setlength{\unitlength}{241.94688692bp}%
    \ifx\svgscale\undefined%
      \relax%
    \else%
      \setlength{\unitlength}{\unitlength * \real{\svgscale}}%
    \fi%
  \else%
    \setlength{\unitlength}{\svgwidth}%
  \fi%
  \global\let\svgwidth\undefined%
  \global\let\svgscale\undefined%
  \makeatother%
  \begin{picture}(1,0.84510593)%
    \lineheight{1}%
    \setlength\tabcolsep{0pt}%
    \put(0,0){\includegraphics[width=\unitlength,page=1]{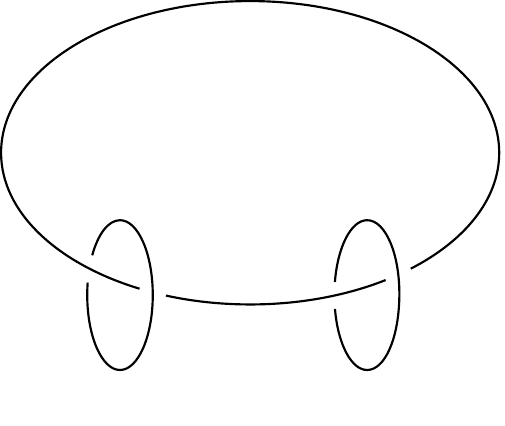}}%
    \put(0.42154882,0.15086174){\color[rgb]{0,0,0}\makebox(0,0)[lt]{\lineheight{1.25}\smash{\begin{tabular}[t]{l}$\cdots$\end{tabular}}}}%
    \put(0.18392724,0.006329){\makebox(0,0)[lt]{\lineheight{1.25}\smash{\begin{tabular}[t]{l}$\frac{a_1}{b_1}$\end{tabular}}}}%
    \put(0.67386865,0.00632891){\makebox(0,0)[lt]{\lineheight{1.25}\smash{\begin{tabular}[t]{l}$\frac{a_n}{b_n}$\end{tabular}}}}%
    \put(0.08348922,0.80003422){\makebox(0,0)[lt]{\lineheight{1.25}\smash{\begin{tabular}[t]{l}$0$\end{tabular}}}}%
  \end{picture}%
\endgroup%

 \caption{Seifert homology sphere $\Sigma(a_1, \ldots, a_n)$ with unnormalized Seifert invariants $(a_1,b_1), \ldots, (a_n, b_n)$.}
  \label{small}
\end{center}
\end{figure}
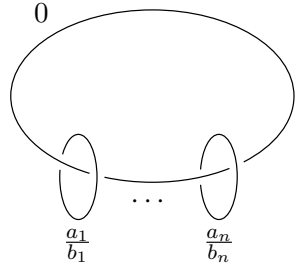 

Alternatively, we can choose representatives $\tilde{b}_j$ for each $b_j$ modulo $a_j$, such that for each $j$ the quantity $r_j = -\frac{\tilde{b}_j}{a_j}$ lies in the interval $(0,1)$, in which case $(b, \tilde{b}_1,\ldots, \tilde{b}_n)$ are uniquely specified and correspond to the {\it normalized} Seifert invariants of $\Sigma(a_1,\ldots, a_n)$. The integer $b$ arising in the normalized situation is an invariant of $\Sigma(a_1,\ldots, a_n)$ that is written $e_0$. In terms of Figure \ref{small}, one applies Rolfsen twists to replace $\frac{a_j}{b_j}$ by $\frac{a_j}{\tilde{b}_j}$for each $j$, so that $\frac{a_j}{\tilde b_j}$ is the unique fraction of the form $\frac{a_j}{b_j + k a_j}$ that is less than $-1$, and then the resulting framing on the large unknot is $e_0$. From \eqref{seifeqn} we find that for the normalized invariants,
\[
\sum_{j=1}^n r_j = -e_0 - \frac{1}{a_1\cdots a_n},
\]
so that $e_0\in \{-1,\ldots, -(n-1)\}$. In terms of unnormalized invariants $(a_j, b_j)$ we have
\begin{equation}\label{e0formula}
e_0 = \sum \lfloor -\ts\frac{b_j}{a_j}\rfloor.
\end{equation}

In the above we have implicitly specified an orientation for $\Sigma(a_1,\ldots, a_n)$, which is the same as the one induced by identifying $\Sigma(a_1,\ldots, a_n)$ with the link of a complex surface singularity. (The oppositely oriented Seifert manifold can be described by a similar procedure, by replacing 1 with $-1$ on the right side of \eqref{seifeqn}.) In particular, with this orientation, a Seifert homology sphere can be realized as the boundary of a negative definite plumbed 4-manifold diffeomorphic to a good resolution of the corresponding  singularity \cite[Theorem 5.2 and Corollary 5.3]{neumannraymond}. A description of this plumbed manifold can be obtained as follows: beginning with the normalized invariants $(e_0, (a_1,\tilde{b}_1),\ldots, (a_n, \tilde{b}_n))$, consider the continued fraction expansion
\[
\frac{a_j}{\tilde{b}_j} = k_{j,1} - \frac{1}{k_{j,2} - \frac{1}{\cdots - \frac{1}{k_{j,\ell_j}}}} = [k_{j,1},\ldots, k_{j, \ell_j}],\]
where each $k_{j,i}\leq -2$. Then $\Sigma(a_1,\ldots, a_n)$ is orientation-preserving diffeomorphic to the boundary of the 4-manifold obtained by plumbing disk bundles according to the graph $\Gamma$ in Figure \ref{plumbfig}. 

We emphasize that in the arguments to follow, the condition that $\Sigma(a_1,\ldots, a_n)$ bound a negative definite plumbing is used repeatedly. In particular, our arguments do not obviously adapt to the case of $-\Sigma(a_1,\ldots, a_n)$, which by \cite[Theorem 5.2]{neumannraymond} does not bound a negative definite plumbing. In several statements to follow we consider ``a Seifert homology sphere $Y = \Sigma(a_1,\ldots, a_n)$'': this phrase simultaneously introduces the symbol $Y$ for the manifold and emphasizes the chosen orientation, since by \cite[Theorem 4.1]{neumannraymond} a Seifert integer homology sphere is uniquely determined up to orientation by the multiplicities $a_1,\ldots, a_n$.

\begin{figure}[h!]
\begin{center}
 \includegraphics[width=7cm]{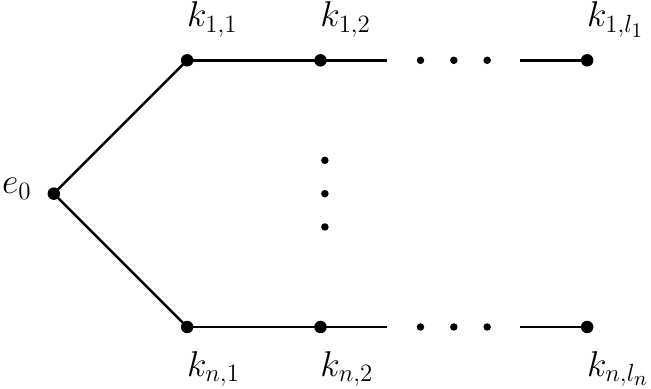}
 \caption{Graph $\Gamma$ describing a plumbing of disk bundles over 2-spheres. The resulting 4-manifold $X_\Gamma$ has boundary the Seifert homology sphere $\Sigma(a_1, \cdots, a_n)$.}
  \label{plumbfig}
\end{center}
\end{figure} 

The structure of $\Sigma(a_1,\ldots, a_n)$ as a Seifert manifold (that is, equipped with a circle action with finite stabilizers) can be seen in Figure \ref{small}. Indeed, surgery along the unknot labeled $0$ gives rise to an $S^1$-bundle over $S^2$ having Euler number $0$, whose fibers appear as meridians to the surgery circle. Dehn surgery along the remaining circles amounts to replacing $n$ regular fibers by exceptional fibers of multiplicities $a_1,\ldots, a_n$. In particular, a generic meridian to the $0$-framed circle is a regular fiber of the Seifert structure, and can be identified with the boundary circle of a disk in the bundle of Euler number $e_0$ appearing in the plumbing of Figure \ref{plumbfig}. 

\subsection{Smooth tau for a regular fiber} Let $K\subset Y = \Sigma(a_1,\ldots, a_n)$ be a knot isotopic to a regular Seifert fiber. The description above of $K$ as the boundary of a disk in a plumbed manifold means that $K$ is {\it slice} in the negative-definite plumbed manifold $X_\Gamma$. This fact constrains the value of $\tau_{sm}(Y,K)$, according to the following result of Raoux \cite{raoux}, generalizing \cite[Theorem 1.1]{OS:hfkgenus}. Again, our formulation is simplified by the assumption that $Y$ is an integer homology sphere.

\begin{theorem}[Raoux \cite{raoux}]\label{raouxthm} Let $X$ be a negative definite 4-manifold with boundary an integer homology sphere $Y$, and $K\subset Y$ a knot. Then for any smooth surface $\Sigma$ properly embedded in $X$ with boundary $K$, we have
\[
\langle c_1(\TT), [\Sigma]\rangle + [\Sigma]\cdot[\Sigma] + 2\tau_{sm}(Y,K) \leq 2g(\Sigma),
\]
for every sharp \spinc structure $\TT$ on $X$.
\end{theorem}

Here $[\Sigma]$ indicates the homology class of $\Sigma$ in $H_2(X, \partial X; \zee)$, identified with $H_2(X;\zee)$ using the fact that $Y$ is a homology sphere, and $[\Sigma]\cdot[\Sigma]$ indicates the intersection product of $[\Sigma]$ with itself. Equivalently, this is the self-intersection number of the closed surface obtained from the union of $\Sigma$ with a Seifert surface for the knot $K$ in $Y$.

A {\it sharp} \spinc structure on a negative-definite 4-manifold with integer homology sphere boundary $Y$ is one with a particular property with respect to the grading on Heegaard Floer homology, which we now describe. The Floer homology $HF^+(Y)$ can be given an integer-valued grading (more generally, there is a grading on $HF^+(Y,\s)$ whenever $c_1(\s)$ is a torsion element of $H^2(Y;\zee)$, though it may take values in $\cue$), and the element $\Theta^+$ described previously is homogeneous with respect to this grading. The degree of $\Theta^+$ is called the $d$-invariant of $Y$, and denoted $d(Y)$ (more generally, one gets a $d$-invariant $d(Y,\s)$ for any \spinc structure on a rational homology sphere; see \cite{OS:grading}). As a concrete case, we have $d(S^3) = 0$.

When $W: Y_1\to Y_2$ is a cobordism between rational homology spheres, the homomorphism on Floer homology $F_{W,\TT}$ is homogeneous for each \spinc structure $\TT$ on $W$, and has degree
\[
\deg(\TT) = \frac{1}{4}( c_1(\TT)^2 - 3\sigma(W) - 2\chi(W)),
\]
where $\sigma(W)$ and $\chi(W)$ are the signature of the intersection form and the (topological) Euler characteristic, respectively. 

Now suppose $X$ is a negative definite 4-manifold with boundary a rational homology sphere $Y$, and assume that $b_1(X) = 0$. By removing a ball from the interior of $X$, we obtain a negative definite cobordism $W: S^3\to Y$. For a \spinc structure $\TT$ on such $W$, the degree formula reduces to 
\[
\deg(\TT) = \frac{1}{4}(c_1(\TT)^2 + b_2(W)),
\]
and in particular this number is the degree of the image $F_{W,\TT}^+(\Theta_{S^3}) \in HF^+(Y)$ for $\Theta_{S^3}$  a generator of degree zero in $HF^+(S^3)$. 

According to \cite[Theorem 1.12]{OS:grading}, when $Y$ is an integer homology sphere, we have the inequality
\begin{equation}\label{dbound}
\deg(\TT) = \frac{1}{4}(c_1(\TT)^2 + b_2(W)) \leq d(Y)
\end{equation}
for each \spinc structure $\TT$ on $W$ (more generally, when $Y$ is a rational homology sphere a similar inequality holds for $d(Y,\s)$, for any \spinc structure $\TT$ extending $\s$).

\begin{definition} A negative definite 4-manifold $X$ having $b_1(X) = 0$ and boundary a homology sphere $Y$ is {\it sharp} if there exists a \spinc structure $\TT$ on $X$ realizing equality in \eqref{dbound}. Any such $\TT$  is said to be a {\it sharp} \spinc structure.
\end{definition}

\begin{proposition} Let $Y = \Sigma(a_1,\ldots,a_n)$ be a Seifert integer homology sphere oriented as the boundary of the negative-definite plumbed 4-manifold $X_\Gamma$ as above, and let $D\subset X_\Gamma$ denote the disk normal to the sphere corresponding to the vertex labeled $e_0$ in Figure \ref{plumbfig}. Assume that the intersection form of $X_\Gamma$ is diagonalizable. Then $X_\Gamma$ supports a sharp \spinc structure $\TT$ with the property that 
\begin{equation}\label{c1bound}
\langle c_1(\TT), [D]\rangle \geq \sqrt{a_1\cdots a_n}.
\end{equation}
\end{proposition}

\begin{remark} If $Y$ is the boundary of an integer homology ball, then the assumption on the intersection form of $X_\Gamma$ is satisfied. Indeed, gluing the homology ball (with appropriate orientation) to $X_\Gamma$, we obtain a smooth, closed, oriented 4-manifold with negative definite intersection form isomorphic to that of $X_\Gamma$. By Donaldson's theorem \cite{donaldson83}, the intersection form must be diagonalizable over $\zee$.
\end{remark}

\begin{proof} Let $X$ be an oriented 4-manifold, either closed or with integer homology sphere boundary, and let $\S(X)$ denote the set of \spinc structures on $X$. Recall (see, e.g., \cite[Section 2.4]{GS}) that there is a natural function $\S(X)\to H^2(X;\zee)$ given by $\TT\mapsto c_1(\TT)$, whose image is the set of ``characteristic elements'' of $H^2(X;\zee)$. By definition, an element $\alpha\in H^2(X;\zee)$ is characteristic if and only if the equality $\alpha(x) = x\cdot x$ holds modulo 2 for each homology class $x\in H_2(X;\zee)$. Furthermore, if $X$ is simply connected (or, more generally, if $H^2(X;\zee)$ contains no elements of order 2), then $\TT\mapsto c_1(\TT)$ is injective. This applies in particular to the manifold $X_\Gamma$ of the Proposition.

Consider a diagonalizing basis $\{e_1,\ldots, e_m\}$ for $H_2(X_\Gamma;\zee)$ with dual basis $\{\epsilon_1,\ldots, \epsilon_m\}$ for $H^2(X_\Gamma;\zee)$ (where $m = b_2(X)$): then $e_i\cdot e_j = -\delta_{ij}$.  For any \spinc structure $\TT$ on $X_\Gamma$ the coefficient of each $\epsilon_j$ in the basis expansion of $c_1(\TT)$ must be odd by the condition that $c_1(\TT)$ is characteristic, hence a \spinc structure maximizes the quantity $\frac{1}{4}(c_1(\TT)^2 + b_2(X_\Gamma))$ if and only if $c_1(\TT) = \sum_j (\pm 1)\epsilon_j$ (and any element of this form is the first Chern class of a unique \spinc structure).

Moreover, by \cite[Corollary 1.5]{OS:plumbing}, the manifold $X_\Gamma$ admits some sharp \spinc structure, and in particular $d(Y) = \max_{\TT} \frac{1}{4}(c_1(\TT)^2 + b_2(X_\Gamma))$. Thus $d(Y) = 0$, and the sharp \spinc structures are exactly those with $c_1(\TT) = \sum_j (\pm 1)\epsilon_j$.

We infer that there are exactly $2^m$ sharp \spinc structures; the corresponding Chern classes are called {\it sharp characteristic vectors.} We must see that \eqref{c1bound} can be realized. Let $\{v_1,\ldots, v_m\}$ be the basis of $H_2(X_\Gamma;\zee)$ represented by the spheres in the plumbing (the zero-sections of the disk bundles). The intersection form $Q$ on $H_2(X)$ can be expressed in this basis by a matrix $(Q_{ij})$, where
\[
Q_{ij} = Q(v_i, v_j) = v_i \cdot v_j.
\]
Thus the off-diagonal entries of $Q$ are 1 or 0 depending whether the corresponding vertices of $\Gamma$ are connected by an edge, and the diagonal entries $Q_{ii}$ are the Euler numbers of the disk bundles (equivalently, the self-intersections of the spheres). We number the vertices so that $v_1$ is the vertex labeled $e_0$ in Figure \ref{plumbfig}.

The element $[D]$ is characterized by the intersection properties $Q([D], v_j) = \delta_{1,j}$. Clearly, then, we can write $[D] = \sum (Q^{-1})_{1i}v_i$. We wish to estimate the maximum value of $\langle \kappa, [D]\rangle$, where $\kappa$ is a sharp characteristic vector. If we write $\kappa = \sum k_j \nu_j$ for $\{\nu_j\}$ the basis hom-dual to $\{v_j\}$, then 
\begin{equation}\label{charpairing}
\langle \kappa, [D]\rangle = \sum_{i,j} k_jQ^{-1}_{1i}\delta_{ij} = \sum_j k_j Q^{-1}_{1j} = Q^*(\kappa, \nu_1),
\end{equation}
where $Q^*$ is the dual intersection form on $H^2(X)$, represented in the basis $\{\nu_j\}$ by the inverse of $Q$.

If we now express the vectors $\kappa$ and $\nu_1$ in the diagonalizing basis $\{\epsilon_j\}$ then the coefficients of $\kappa$ are all $\pm 1$, and the maximum value obtained on the right of \eqref{charpairing} is clearly the $L^1$-norm of $\nu_1$ in this basis:
\[
\max_{\kappa\mbox{ \scriptsize sharp}} \langle \kappa, [D]\rangle = \|\nu_1\|_{L^1(\epsilon_j)}.
\]

Consider the transition matrix $B$ between the bases $\{\nu_j\}$ and $\{\epsilon_j\}$, i.e., 
\[
\nu_k = \sum_j B_{jk} \epsilon_j.
\]
Thus the $k$th column of $B$ gives the coefficients of $\nu_k$ in terms of $\{\epsilon_k\}$, and we are interested in estimating the $L^1$-norm of the first column.

In terms of matrices, we have that $B^{-1}$ diagonalizes the dual form $Q^*$:
\[
(B^{-1})^T Q^* B^{-1} = -{\mathbb I},
\]
or equivalently $Q^* = -B^T B$. It follows that the $(1,1)$ entry of the matrix $Q^*$ inverse to the intersection form $Q$ (in the basis $\{v_j\}$) is minus the square of the $L^2$-norm of the first column of $B$. The proposition follows from two observations:
\begin{enumerate}
\item The $L^2$ norm is a lower bound on the $L^1$ norm, and
\item The $(1,1)$ entry of the inverse of the intersection matrix $Q$ is $-a_1\cdots a_n$.
\end{enumerate}
Both of these are elementary; the first is the triangle inequality and the second can be seen, for instance, from Cramer's rule.

\end{proof}

Observe that in the notation of the above proof, the self-intersection $[D]\cdot [D]$ is exactly equal to the $(1,1)$ entry of $Q^*$, so that $[D]\cdot [D] = -a_1\cdots a_n$. With this and using the \spinc structure obtained in the Proposition, Theorem \ref{raouxthm} gives:

\begin{corollary}\label{smtaucor} If $Y = \Sigma(a_1,\ldots, a_n)$ is a Seifert integer homology sphere that bounds an integer homology ball (or more generally if the plumbed manifold $X_\Gamma$ has diagonalizable intersection form), then a knot $K$ isotopic to a regular fiber in the Seifert structure satisfies
\[
2\tau_{sm}(Y,K) \leq a_1\cdots a_n - \sqrt{a_1\cdots a_n}.
\]
\end{corollary}

\subsection{Contact tau for a regular fiber}\label{estimatesec2} Let $Y$ be an integer homology sphere, and consider a contact structure $\xi$ on $Y$. We imagine that $\xi$ is induced by a symplectically convex domain bounded by $Y$, so in particular we may assume that $\xi$ is strongly symplectically fillable. This implies that $\hat{c}(\xi)$ is nonzero, so the invariant $\tau_\xi(Y,K)$ for knots $K\subset Y$ is defined.

Recall that any knot in a contact manifold is smoothly isotopic to (many different) {\it Legendrian} knots, which are knots everywhere tangent to the contact distribution. If $K$ is a Legendrian, then the contact planes provide a framing of $K$, and if $K$ is nullhomologous then a (choice of) Seifert surface for $K$ defines a second framing, the Seifert framing. The difference between the contact framing and the Seifert framing is called the {\it Thurston-Bennequin invariant} of $K$ and is written $\tb(K)$. Having chosen a Seifert surface $S$ for $K$, the {\it rotation number} $\rot(K)$ of $K$ is defined to be the relative Euler number of $\xi$ restricted to $S$, relative to the trivialization of $\xi$ over $\partial S = K$ provided by the tangents to $K$. We have the following fundamental inequality due to Hedden, generalizing a result of Plamenevskaya \cite{plam04} for knots in $S^3$.

\begin{theorem}[Hedden \cite{heddentau}]\label{heddenthm} Let $(Y,\xi)$ be a contact 3-manifold with $\hat{c}(\xi)\neq 0$. Then for any nullhomologous Legendrian knot $K$ in $Y$ we have
\[
\tb(K) + |\rot(K)| \leq 2\tau_\xi(Y,K) - 1.
\]
\end{theorem}

Now suppose $Y$ is a Seifert manifold. A regular Seifert fiber in $Y$ inherits a natural framing from the Seifert structure, called the {\it fiber framing}: indeed, thinking of a framing of a knot as specified by an equivalence class of pushoff, a the fiber framing of a regular fiber is given by a second nearby regular fiber. A knot isotopic to a regular fiber receives, once such an isotopy is chosen, a fiber framing as well.

\begin{definition} 
Let $\xi$ be a tight contact structure on a Seifert 3-manifold, $L$ a Legendrian knot smoothly isotopic to a regular fiber, and $\phi$ a choice of such an isotopy. The {\it twisting number} $\tw(L, \phi))$ is the difference between the contact framing and the fiber framing induced by $\phi$. The {\it maximal twisting number}  $\tw(\xi)$ of $\xi$ is 
\[
\tw(\xi) = \sup_{L, \phi}\{\tw(L,\phi)\},
\]
the maximum taken over all such Legendrians $L$ and all choices of isotopy $\phi$.
\end{definition}

It can be seen that on a Brieskorn sphere $\tw(L, \phi)$ is independent of $\phi$, see \cite{Ghiggini:seifert}. Moroever if $\xi$ is a tight contact structure then $\tw(\xi)$ is finite.

\begin{lemma} For a Seifert homology sphere $Y = \Sigma(a_1,\ldots,a_n)$, the fiber framing of a regular fiber is equal to $a_1\cdots a_n$ (when measured with respect to the Seifert framing).
\end{lemma}
\begin{proof} It is elementary that if $K$ is a knot in an integer homology sphere $Y$, then the manifold obtained by $p$-framed surgery along $K$ (with respect to the Seifert framing) has first homology of order $|p|$. In the surgery diagram of Figure \ref{small}, the fiber-framed regular fiber is represented by a meridian of the circle labeled $0$, with framing 0 (in the diagram). Performing surgery on this meridian with fiber framing therefore cancels the $0$-framed circle and leaves a surgery diagram for a connected sum of lens spaces $L(a_1,b_1)\#\cdots\# L(a_n,b_n)$. The latter 3-manifold has first homology of order $a_1\cdots a_n$, which proves the lemma up to sign. That the fiber framing is positive can be seen, for example, by performing $+1$-framed surgery on the meridian (framing measured in the diagram), and noting that the resulting manifold has first homology of order $a_1\cdots a_n + 1$. For the latter calculation, one can use the fact that a 3-manifold described by the surgery diagram of Figure \ref{small}, with arbitrary $a_k$ and $b_k$, has first homology of order equal to the absolute value of  $a_1\cdots a_n\sum_k \frac{b_k}{a_k}$ (indeed, this quantity is easily seen to equal the absolute value of the determinant of a presentation matrix for the first homology; see \cite[Section 1.1.4]{savelievbook}, compare equation \eqref{seifeqn}). In this light, performing $+1$ surgery on an additional meridian amounts to replacing $(\frac{a_1}{b_1},\ldots, \frac{a_n}{b_n})$ with  $(\frac{a_1}{b_1},\ldots, \frac{a_n}{b_n}, 1)$, from which the result immediately follows. Note, of course, that this depends on our chosen orientation for $Y$.
\end{proof}

\begin{corollary}\label{cor:contacttau} Let $Y = \Sigma(a_1,\ldots, a_n)$, $\xi$ a contact structure on $Y$ with $\hat{c}(\xi) \neq 0$, and $K$ a knot smoothly isotopic to a regular fiber. Then
\[
2\tau_\xi(Y,K) \geq \tw(\xi) + a_1\cdots a_n  + 1.
\]
\end{corollary}

\begin{proof} We may suppose that $K$ is chosen to be Legendrian with twist number equal to $\tw(\xi)$. Then this is Theorem \ref{heddenthm}, after discarding the nonnegative $|\rot(K)|$ and observing that for this choice of $K$, 
\begin{eqnarray*}
\tb(K) &=& \mbox{(contact framing)}-\mbox{(Seifert framing)} \\
&=& \mbox{(contact framing)}-\mbox{(fiber framing)} + \mbox{(fiber framing)}-\mbox{(Seifert framing)}\\
&=& \tw(K,\phi) + a_1\cdots a_n\\
&=& \tw(\xi) + a_1\cdots a_n.
\end{eqnarray*}
\end{proof}

Combining this with Corollary \ref{smtaucor} gives:

\begin{corollary}\label{taudiffcor} Suppose $Y = \Sigma(a_1,\ldots, a_n)$ is a Seifert integer homology sphere that bounds an integer homology ball (or assume just that the plumbed manifold $X_\Gamma$ has diagonalizable intersection form), and let $\xi$ be a contact structure on $Y$ with $\hat{c}(\xi) \neq 0$. Then there exists a knot $K\subset Y$ (isotopic to a regular fiber) such that 
\[
2(\tau_\xi(Y,K) - \tau_{sm}(Y,K)) \geq \tw(\xi) + \sqrt{a_1\cdots a_n} + 1.
\]
\end{corollary}

\begin{corollary}\label{ratboundcor} Let $Y = \Sigma(a_1,\ldots, a_n)$ be a Seifert integer homology sphere. If $\xi$ is a contact structure on $Y$ with $\tw(\xi) \geq -\sqrt{a_1\cdots a_n}$, then $(Y,\xi)$ does not admit a contact type embedding in $(\arr^4,\omega_{std})$.
\end{corollary}
\begin{proof} By the remarks after Corollary \ref{reducedcor}, if $(Y,\xi)$ admits a contact type embedding then it is contactomorphic to the boundary of a symplectically convex domain in $\arr^4$ with the homology of a ball. Hence we may assume that $c^+(\xi) \neq 0$, and that $Y$ bounds a homology ball. Then the given inequality combined with the previous corollary shows that there is a knot $K$ in $Y$ with $\tau_\xi(Y,K) - \tau_{sm}(Y,K) > 0$. By Corollary \ref{convexcor}, $(Y,\xi)$ does not bound a symplectically convex domain.
\end{proof}

\section{Twisting numbers for Seifert homology spheres}\label{twistsec}

\bigskip

In light of Corollary \ref{ratboundcor}, our goal is now to derive a lower bound on the twisting number. To achieve this we apply certain techniques standard in the classification problem for contact structures on Seifert manifolds \cite{Honda:classification1,GS:classification, EH:nonexistence}.

Consider the Seifert homology sphere $Y=\Sigma(a_1, \ldots, a_n)$ as in Figure~\ref{small}, so that $(a_1, b_1), \ldots, (a_n, b_n)$ are the unnormalized Seifert invariants, satisfying \eqref{seifeqn} with $b = 0$. We first spell out topological conventions that are standard in this situation and are implicit in Figure \ref{small}.

Let $F_1, F_2, \ldots, F_n$ be the exceptional Seifert fibers of $Y$, and for $i = 1,\ldots, n$ let $V_{i}= D^2\times S^1$ denote a fixed solid torus. In particular we identify $\partial V_{i}$ with $\mathbb{R}^2/\mathbb{Z}^2$ such that $(1,0)^{T}$  corresponds to the meridian $\partial D^2 \times pt$, and $(0,1)^T$  the longitude $pt \times S^1$. On the other hand, let $S$ be the complement of $n$ disjoint open disks embedded in the 2-sphere: then each component of the boundary of $S\times S^1$ can then be identified with $\arr^2/\zee^2$ such that $(0,1)^{T}$ corresponds to $pt\times S^1$, and $(1,0)^{T}$ is identified with a component of $-\partial({pt.} \times S)$. We choose the orientation so that the standard orientation of $\arr^2/\zee^2$ is identified with the opposite of the induced orientation on $\partial (S\times S^1)$.

There is then a diffeomorphism
\begin{equation}\label{gluing}
Y = \Sigma(a_1,\ldots, a_n) \cong (S\times S^1)\cup_\partial (V_1\sqcup \cdots \sqcup V_n),
\end{equation}
where $V_i$ is attached to the $i$-th component $\partial_i(S\times S^1)$ by maps $A_{i}:\partial V_{i}\rightarrow -\partial_i(S\times S^1)$ given by 
\[
A_i=\left( \begin{array}{cc}
a_i&u_i\\ b_i&v_i
\end{array} \right).
\]
Here $u_i$, $v_i$ are integers that may be chosen arbitrarily so long as the resulting $A_i$ is an orientation-preserving homemorphism, the ambiguity in the choice reflecting the freedom in choosing the identification between a neighborhood of $F_i$ and the solid torus $V_i = D^2\times S^1$.  The identification \eqref{gluing} depends on this choice, and hence so does the numerical value of the twisting number (for Legendrian representatives of the exceptional fibers) defined below. To eliminate the ambiguity, we specify $u_i$, $v_i$ uniquely by the requirements $a_iv_i-b_iu_i=1$ and $0 < u_i<a_i$ for each $i$. Note that the case $u_i = 0$ does not arise since each $a_i \geq 2$.

Now suppose that $Y$ is equipped with a contact structure $\xi$. Then we can isotope each singular fiber $F_i$ to be Legendrian, and by stabilizing Legendrian representatives we may assume the contact framing of $F_i$ is arbitrarily small. The identification \eqref{gluing} provides a framing of $F_i$ when the latter is identified with the core of the solid torus $V_i$, and the contact framing of a Legendrian representative of $F_i$ compared to this framing is called the {\it twisting number} of the representative. Taking the twisting number to be $k_i<0$, we can then adjust \eqref{gluing} by an isotopy so that $V_{i}$  is a standard contact neighborhood of $F_i$, having convex boundary and dividing set $\Gamma_{\partial V_i}$ consisting of two simple closed curves of slope $\frac{1}{k_i}$. 

With our framing conventions above, slope $\frac{1}{k_i}$ on $\partial V_i$ corresponds to the vector $(k_i,1)^T$, so that each component of $\Gamma_{\partial V_i}$ is homotopic to a concatenation of one longitude and $k_i$ meridians. When measured in $-\partial_i(S\times S^1)$, that is after multiplying the matrices $A_i$ with the vectors  $(k_i,1)^T$, these slopes become

\begin{equation}\label{slopeform}
s_i=\frac{b_ik_i+v_i}{a_ik_i+u_i},~~ i=1, \ldots, n
\end{equation}
The denominator of $s_i$ is never zero by our choice of $u_i$, so that the dividing slope on $-\partial_i(S\times S^1)$ is not infinite. Hence, according to the flexibility theorem of Giroux, we can arrange by a small isotopy of the contact structure near $\partial V_i$ that the characteristic foliation on $-\partial_i(S\times S^1)$ is by parallel circles of infinite slope, which is to say that $-\partial_i(S\times S^1)$ has a Legendrian ruling by ``vertical'' circles isotopic to $pt\times S^1$.

We wish to apply the following ``twist number lemma'' of Honda \cite{Honda:classification1} to the Legendrian exceptional fibers $F_i$.

\begin{lemma}\label{tnl}
Consider a Legendrian curve $L$, and suppose that with respect to some framing the twisting number of $L$ is $k$. Let $V$ be a standard neighborhood of $L$, so that in the given framing $\partial V$ is convex with two dividing curves having slope $\frac{1}{k}$. If there exists a bypass $D$ which is attached to 
$\partial V$ along a Legendrian ruling curve of slope $r$, and $\frac{1}{r} \geq k + 1$, then there exists a Legendrian curve with twisting number $k+1$ smoothly isotopic to $L$.
\end{lemma}

In our situation the ruling curves have infinite slope on $-\partial_i(S\times S^1)$, corresponding to the vector $(0,1)^T$. Passing to $\partial V_i$ using $A_i^{-1}$, the ruling slope is then given by the vector $(-u_i, a_i)$, i.e. $r_i = -\frac{a_i}{u_i}$. Since $\frac{1}{r_i} = -\frac{u_i}{a_i} \in (-1,0)$, we see that as long as suitable bypasses can be found, we can replace $F_i$ by smoothly isotopic Legendrians having sequentially larger twist numbers, implicitly adjusting the neighborhoods to be convex with infinite ruling slope as before, until the twist number $k_i$ reaches $-1$. 
 
 \begin{lemma}[Congruence Principle] Let $\xi$ be a tight contact structure on a Seifert manifold $Y$ as in \eqref{gluing}, and suppose there exists a Legendrian regular fiber in $(Y, \xi)$ having twist number $t<0$. Then for any pair $i,j$, there exist Legendrian representatives of the exceptional fibers $F_i$, $F_j$ with twist numbers $k_i,k_j <0$, such that either:
 \begin{enumerate}
 \item One of $k_i$ or $k_j$ is equal to $-1$, or
 \item There is an integer $d$ with $t\leq d < 0$ such that $d = u_i$ modulo $a_i$ and $d = u_j$ modulo $a_j$, and there is a Legendrian regular fiber in $(Y,\xi)$ having twist number $d$. Moreover, in this case 
 \[
 d = a_ik_i + u_i = a_j k_j + u_j.
 \]
 \end{enumerate}
 \end{lemma}
 
\begin{proof} Consider a vertical annulus $\A_{i,j}$ between $-\partial_i(S\times S^1)$ and $-\partial_j(S\times S^1)$, which is to say that $\A_{i,j}$ is isotopic to an annulus of the form $a\times S^1$ for an arc $a$ between boundary components $i$ and $j$ of $S$. In particular, we take the boundary circles $\partial_i\A_{i,j}$ and $\partial_j\A_{i,j}$ of $\A_{i,j}$ to be Legendrian ruling curves on the boundary tori. Since the twisting numbers of the boundary circles are negative, we can make $\A_{i,j}$ convex by a small isotopy. In more detail, recall that the twisting number $\tw(L, \Sigma)$ of a Legendrian $L$ lying on a surface $\Sigma$ is equal to the difference between the framing on $L$ induced by the contact structure and that induced by $\Sigma$ (equivalently, the framing induced by a vector field along $L$ normal to the contact structure or the surface, respectively). If a compact orientable surface has Legendrian boundary where each boundary circle has nonpositive twisting number, then \cite[Proposition 3.1]{Honda:classification1} shows that the surface can be made convex by a small perturbation fixing the boundary. In the case at hand the framing induced by the annulus is the same as that induced by the boundary torus, so that $\tw(\partial_i\A_{i,j}, \A_{i,j}) = \tw(\partial_i\A_{i,j}, -\partial_i(S\times S^1))$. Observe that  $-\partial_i(S\times S^1)$ is already convex, and recall that on a convex surface the twisting number of a Legendrian $L$ in general position is given by $-\frac{1}{2}$ times the number of geometric intersections of $L$ with the dividing set. This already shows the twisting numbers of the boundary circles of $\A_{i,j}$ are nonpositive, so $\A_{i,j}$ can be made convex. We can also suppose that $\A_{i,j}$ is chosen to contain a curve $c$ isotopic to $pt\times S^1$ having a Legendrian realization isotopic to the regular fiber of twist number $t$ as in the statement.

To understand the twisting numbers of the boundary circles of $\A_{i,j}$ concretely, recall that in our case the dividing set on one boundary torus consists of two parallel circles, and the Legendrian ruling curves intersect them minimally.  Moreover, a dividing curve on $-\partial_i(S\times S^1)$ has slope $s_i$. The algebraic intersection number between a curve of slope $s_i$ and the infinite slope ruling curve is given by $a_ik_i + u_i$, which is negative and thus equal to the twisting number, and similar for $s_j$. 

Now, the dividing set on $\A_{i,j}$ is a disjoint union of properly embedded arcs (there are no closed components since $\xi$ is tight, using Giroux's criterion \cite[Theorem 3.5]{Honda:classification1}), whose endpoints on the boundary alternate with points of intersection between the boundary and the dividing set on $-\partial_{i}(S\times S^1)$ (or $-\partial_j(S\times S^1)$). Moreover, since $\A_{i,j}$ contains a vertical arc with twist number $t$, we can assume that at most $|2t|$ dividing curves of $\A_{i,j}$ connect different boundary components of the annulus. Hence if either $|a_ik_i + u_i|$ or $|a_jk_j + u_j|$ is greater than $|t|$, there exists a dividing arc on $\A_{i,j}$ that connects two points on the same boundary component, and an innermost such arc determines a bypass for the corresponding torus. (This is essentially Honda's ``Imbalance Principle.'') Attaching these bypasses sequentially and applying the twist number lemma, we can suppose that either $k_i$ can be increased to $-1$ or that $k_i$ satisfies $t \le a_ik_i + u_i < 0$, and similarly for $j$.

Now, if at this point we have $a_ik_i + u_i \neq a_jk_j + u_j$, then by the imbalance principle there exists a bypass on one side or the other of the annulus. Thus, as long as $a_ik_i + u_i \neq a_jk_j +u_j$, and $k_i, k_j < -1$, we can add bypasses to either $\partial V_i$ or $\partial V_j$ and thereby assume that the new solid tori are standard neighborhoods of Legendrian exceptional fibers of increasing twist number. 

By continuing this process, either one of $k_i$, $k_j$ becomes $-1$, or we reach a point at which $a_ik_i + u_i = a_jk_j + u_j$, which in particular says that the number $d = a_ik_i + u_i$ satisfies the two congruences $d = u_i\mod a_i$ and $d = u_j \mod a_j$. In the latter case, a ruling curve on either $\partial_i(S\times S^1)$ or $\partial_j(S\times S^1)$ is isotopic to a regular fiber, and is a Legendrian with twist number $d$ as desired.


\end{proof}

Observe that in the situation at the end of the preceding proof, it may happen that the intersection numbers $a_ik_i + u_i$ and $a_j k_j + u_j$ are equal, yet there is still a dividing arc on $\A_{i,j}$ with endpoints on the same boundary component. In this case there must actually be bypasses on both sides of $\A_{i,j}$, and hence both $k_i$ and $k_j$ can be increased. The process then resumes until one of $k_i$, $k_j$ reaches $-1$, or the congruence of part (2) of the Lemma is again realized, and there are no further bypasses on $\A_{i,j}$.

\begin{definition} Let $Y$ be a Seifert manifold with singular fibers $F_1,\ldots, F_n$, and $I\subset \{1,\ldots, n\}$. A collection of Legendrian representatives for $\{F_i\}_{i\in I}$ is {\em twist-balanced}, or simply {\it balanced}, if for each $i,j\in I$ and any convex vertical annulus $\A_{i,j}$ connecting standard neighborhoods of $F_i$ and $F_j$ avoiding the singular fibers, there are no boundary-parallel dividing arcs on either side of $\A_{i,j}$. 
\end{definition}

Applying the congruence principle repeatedly yields:

\begin{corollary}\label{balancecor} For any subset $\{F_i\}_{i\in I}$ of singular fibers in the Seifert manifold $Y$, either there exists $i \in I$ and a Legendrian representative of  $F_i$ having twist number $-1$, or $\{F_i\}_{i\in I}$ admit twist-balanced representatives.

In the twist-balanced case, the balanced representatives of $F_i$ have twist numbers $k_i$ satisfying 
\[
a_i k_i + u_i = a_j k_j + u_j
\]
for all $i, j\in I$. If there is a Legendrian regular fiber in $Y$ having twist number $t<0$, then the common value $d = a_ik_i + u_i$ can be taken to satisfy $t\leq d < 0$.
\end{corollary}

\begin{proposition}\label{mainprop} Let $Y = \Sigma(a_1,\ldots, a_n)$ be a standardly-oriented Brieskorn homology sphere with a given tight contact structure, let $I = \{1,\ldots, n-1\}$, and assume that $\{F_i\}_{i\in I}$ admit twist-balanced Legendrian representatives. 
Then there exists a Legendrian regular fiber in $Y$ having twist number $-a_1\cdots a_{n-1}$. 
\end{proposition}

From Corollary \ref{balancecor}, in the situation of the Proposition we also have at least one of the following:
\begin{enumerate}
\item[(a)] The fiber $F_n$ admits a Legendrian representative of twist number $-1$, or
\item[(b)] Some $F_i$, $1\leq i \leq n-1$, admits a representative of twist number $-1$, or
\item[(c)] The collection of all singular fibers $\{F_i\}_{i=1}^n$ admit twist-balanced representatives, with common value $d \geq -a_1\cdots a_{n-1}$.
\end{enumerate}

\begin{proof} For later use, we consider a collection of $j\geq 2$ twist-balanced fibers $\{F_1, \ldots, F_j\}$ (for the Proposition, we will take $j = n-1$). The twist-balanced condition means that for appropriate Legendrian representatives of $F_1,\ldots, F_{j}$, their twist numbers $k_i$ satisfy $a_1k_1 + u_1 = \cdots = a_{j}k_{j}  + u_{j} =: d$. Moreover, we can find a pairwise disjoint collection of convex vertical annuli $\A_{i, i+1}$, $i = 1,\ldots, j-1$, so that $\A_{i, i+1}$ connects neighborhoods of $F_i$ and $F_{i+1}$, and contains no boundary-parallel dividing curves. Writing $V_i$ for the standard neighborhood of $F_i$, the boundary of a regular neighborhood of the union of $V_1,\ldots, V_{j}$ with these annuli becomes, after smoothing corners, a convex torus $T_{CR}$. Here ``CR'' stands for ``cut-and-round,'' which is a standard term for this procedure.

The slope of the dividing curves on $T_{CR}$ can be calculated with the help of the ``edge-rounding lemma'' of Honda \cite[Lemma 3.11]{Honda:classification1}: writing $s_i$ as in \eqref{slopeform} for the slope of $-\partial_j(S\times S^1)$, the slope of $T_{CR}$ is given by
\[
s(T_{CR}) = s_1 + \cdots + s_{j} - \ts\frac{j-1}{d},
\]
where the last term arises since each rounding of a corner contributes a slope of $-\frac{1}{4d}$ and there are $4(j-1)$ corners to round (for more on this procedure and the calculation of the resulting slope, we refer to \cite[Lemma 4.5]{GS:classification} or \cite[Lemma 7]{EH:nonexistence}). We calculate:
\begin{eqnarray*}
s(T_{CR}) &=& \ts\frac{1}{d}( b_1k_1 + v_1 + \cdots + b_{j}k_{j} + v_{j} - j+1) \\
&=& \ts\frac{1}{d}\left(\frac{b_1}{a_1}(d-u_1) + \cdots + \frac{b_{j}}{a_{j}}(d - u_{j}) + v_1 + \cdots + v_{j} - j +1\right) \\ 
&=& \ts\frac{b_1}{a_1} + \cdots + \frac{b_{j}}{a_{j}}  + \frac{1}{d}\left( \frac{1}{a_1} + \cdots + \frac{1}{a_{j}} - j+1\right)
\end{eqnarray*}
where we have used that $a_iv_i - b_iu_i = 1$. Since $j\geq 2$, and the $a_i$ are relatively prime and at least 2, it is easy to see that in the last line the term in parentheses is strictly negative. Since each $k_i \leq -1$, $d = a_i k_i + u_i$ is also negative and we infer
\[
s(T_{CR}) > \frac{b_1}{a_1} + \cdots + \frac{b_{j}}{a_{j}} .
\]

Now we take $j = n-1$. Using \eqref{seifeqn} with $b = 0$ we have that 
\begin{equation}\label{slopeineq}
\frac{b_1}{a_1} + \cdots + \frac{b_{n-1}}{a_{n-1}} = \frac{1}{a_1\cdots a_n}-\frac{b_n}{a_n} > - \frac{b_nk_n + v_n}{a_nk_n + u_n} = -s_n,
\end{equation}
where it is easy to check that the inequality holds for any sufficiently negative $k_n$. This last fraction is just the slope of the boundary of a standard neighborhood of a Legendrian representative of $F_n$ having twist number $k_n$, measured in the coordinates on $\partial_n(S \times S^1)$, with a change in sign to account for the orientation change. 

It follows that the region between $T_{CR}$ and $\partial_n(S\times S^1)$, being diffeomorphic to $T_{CR}\times [0,1]$ with  convex boundaries having slopes $s(T_{CR})$ and $-s_n$, contains a boundary-parallel convex torus $T'$, having slope $\frac{b_1}{a_1} + \cdots + \frac{b_{n-1}}{a_{n-1}}$. A dividing curve on this torus intersects the vertical ruling curve (a Legendrian regular fiber) $a_1\cdots a_{n-1}$ times, hence this vertical Legendrian has twist number $-a_1\cdots a_{n-1}$.
\end{proof}

\begin{theorem}\label{twistthm} Let $Y= \Sigma(a_1,\ldots, a_n)$ be a standardly-oriented Brieskorn homology sphere, with contact structure $\xi$. Let $F_1,\ldots, F_n$ be the multiple fibers in the Seifert structure, having multiplicities $a_1,\ldots, a_n$. Then the twist number of $\xi$ satisfies
\[
\tw(\xi) \geq -\sqrt{a_1\cdots a_n}.
\]
\end{theorem}

\begin{proof} We can assume that $\tw(\xi)\leq 0$.

We have seen that if $F_i$ admits a Legendrian representative with twist number $k_i<0$, then a vertical ruling curve on a convex neighborhood of $F_i$ intersects a dividing curve $|a_i k_i + u_i|$ times. In particular, since the ruling curves are Legendrians isotopic to a regular fiber, we have $\tw(\xi) \geq a_i k_i + u_i$. 

Suppose that there exist distinct $i,j\in\{1,\ldots, n\}$ such that both $F_i$ and $F_j$ have representatives with twist number $-1$. Then $\tw(\xi)$ is greater than or equal to both $-a_i + u_i$ and $-a_j + u_j$, hence $\tw(\xi)$ is strictly larger than both $-a_i$ and $-a_j$. Therefore $\tw(\xi)^2 < a_i a_j < a_1\cdots a_n$ as required. 

If there is exactly one multiple fiber that admits a representative with twist number $-1$, renumber the fibers so that $F_n$ admits such a representative. Then it must be that $\{F_1,\ldots, F_{n-1}\}$  admit twist-balanced representatives, and from Proposition \ref{mainprop} there exists a Legendrian regular fiber having twist number $-a_1\cdots a_{n-1}$ and another having twist number greater than $-a_n$ (since $F_n$ has twist number $-1$). Therefore $\tw(\xi)^2 <(a_1\cdots a_{n-1})(a_n)$,
giving the result in this case.

Now suppose that none of the multiple fibers admit representatives of twist number $-1$. By Corollary \ref{balancecor} and Proposition \ref{mainprop}, the multiple fibers admit a set of twist-balanced representatives with twist numbers $k_i$, $i = 1,\ldots, n$, having value $d = a_ik_i + u_i$ with $d\geq -a_1\cdots a_{n-1}$. 

First let us suppose $n \geq 4$, and let $2\leq j \leq n-2$ so that the collections $\{F_1,\ldots, F_j\}$ and $\{F_{j+1},\ldots, F_n\}$ both contain at least two fibers. We apply the cut-and-round procedure to each of these two collections as in the proof of Proposition \ref{mainprop}, to obtain two disjoint tori $T_{CR}$ and $T'_{CR}$, enclosing the first and second collections respectively. From the calculation in that proof, we have the sequence of inequalities
\begin{eqnarray*}
s(T_{CR}) &=& \ts\frac{b_1}{a_1} + \cdots + \frac{b_{j}}{a_{j}}  + \frac{1}{d}\left( \frac{1}{a_1} + \cdots + \frac{1}{a_{j}} - j+1\right) \\
&>& \ts\frac{b_1}{a_1} + \cdots + \frac{b_{j}}{a_{j}}\\
&=& \ts \frac{1}{a_1\cdots a_n} - \frac{b_{j+1}}{a_{j+1}} - \cdots - \frac{b_n}{a_n}\\
&>& \ts - \frac{b_{j+1}}{a_{j+1}} - \cdots - \frac{b_n}{a_n} \\
&>& \ts - \frac{b_{j+1}}{a_{j+1}} - \cdots - \frac{b_n}{a_n} - \frac{1}{d}\left( \frac{1}{a_{j+1}} + \cdots + \frac{1}{a_{n}} - n+j +1\right)\\ 
&=& -s(T'_{CR}).
\end{eqnarray*}
In this caculation we have used \eqref{seifeqn} along with the fact that both $d$ and the quantities in parentheses above are negative. We can choose $T_{CR}$ and $T_{CR}'$ so that they cobound a subset of $Y$ diffeomorphic to $T_{CR}\times [0,1]$ with $T_{CR}\times \{1\}$ identified with $-T'_{CR}$. From the calculation above, it follows that this region contains boundary-parallel tori $T_1$ and $T_2$ with 
\[
s(T_1) = \ts\frac{b_1}{a_1} + \cdots + \frac{b_{j}}{a_{j}}, \quad\quad s(T_2) =  \ts - \frac{b_{j+1}}{a_{j+1}} - \cdots - \frac{b_n}{a_n}.
\]
Since these slopes have denominators $a_1\cdots a_j$ and $a_{j+1}\cdots a_n$ respectively, vertical ruling curves on $T_1$ and $T_2$ have twist numbers $-a_1\cdots a_j$ and $-a_{j+1}\cdots a_n$ respectively. Therefore $\tw(\xi)^2 \leq (a_1\cdots a_j)(a_{j+1}\cdots a_n)$, completing the argument when $n\geq 4$.

For the case $n =3$ we proceed as follows. We are assuming that $F_1, F_2$, and $F_3$ admit twist-balanced representatives with twist numbers $k_i$ satisfying
\begin{equation}\label{ddef}
0 > a_1k_1 + u_1 = a_2 k_2 + u_2 = a_3 k_3 + u_3 = d \geq -a_1a_2.
\end{equation}
In particular $d$ is congruent to $u_i$ mod $a_i$ for $i = 1,2$, hence we have $d = -a_1a_2 + m$ for a unique positive integer $m$ between $1$ and $a_1a_2 - 1$. 

Observe that equation \eqref{seifeqn} implies, after distributing $a_1\cdots a_n$, setting $b =0$ and reducing modulo $a_i$, that  $b_ia_1\cdots \hat{a}_i\cdots a_n = 1$ mod $a_i$ for each $i$. Similarly, the condition $a_i v_i - b_i u_i = 1$ shows $-b_i u_i = 1$ mod $a_i$, hence 
\begin{equation}\label{ucond}
u_i = -a_1\cdots \hat{a}_i\cdots a_n \mod a_i.
\end{equation}
In our situation with $n= 3$, we have
\[
a_3k_3 + u_3 = d = -a_1a_2 + m,
\]
and since $u_3 = -a_1a_2$ mod $a_3$ it follows that $m = 0$ mod $a_3$. Hence $m = ca_3$ for some positive integer $c$. Furthermore, from \eqref{ddef}, \eqref{ucond}, and $d = -a_1a_2 + m$, we infer that 
\[
ca_3 = m = -a_2a_3 \mod a_1 \quad \mbox{and} \quad ca_3 = m = -a_1a_3 \mod a_2.
\]
Since the $a_j$ are relatively prime, it follows that $c = -a_2$ mod $a_1$ and $c= -a_1$ mod $a_2$. Hence we can write 
\begin{equation}\label{ccond}
c = \ell_1 a_1 - a_2 \quad\mbox{and}\quad c = \ell_2 a_2 - a_1
\end{equation}
for some integers $\ell_1, \ell_2$, which must be positive since $c$ is. These then satisfy
\[
\ell_1a_1 + a_1 = \ell_2 a_2 + a_2
\]
and using relatively prime again, we have $\ell_1 + 1 = 0$ mod $a_2$ and $\ell_2 + 1 = 0$ mod $a_1$. Thus
\[
\ell_1 + 1 = x_1a_2 \quad\mbox{and}\quad \ell_2 + 1 = x_2 a_1
\]
for positive integers $x_1$, $x_2$. Using this to substitute for $\ell_1$ in \eqref{ccond} we have
\[
c = (x_1a_2 - 1)a_1 - a_2.
\]
Hence, 
\begin{eqnarray*}
m - a_1a_2 &=& ca_3 - a_1 a_2\\
&=& x_1a_1a_2a_3 - a_1 a_3 - a_2 a_3 - a_1 a_2 \\
&=& \ts a_1a_2a_3(x_1 - \frac{1}{a_1} -  \frac{1}{a_2} - \frac{1}{a_3}).
\end{eqnarray*}
Since $m < a_1a_2$, the left side is negative. On the other hand, because $x_1$ is a positive integer, the only situation in which the right side is negative is when $x_1 = 1$ and $\{a_1,a_2, a_3\} = \{2,3,5\}$. This shows that when there are 3 singular fibers, the twist-balanced situation arising in (c) of the remark after Proposition \ref{mainprop}, i.e., the case where no multiple fibers have representatives of twisting $-1$, can only possibly occur for the Poincar\'e homology sphere $\Sigma(2,3,5)$. For all other $\Sigma(a_1,a_2,a_3)$ we must end in one of the cases considered previously, in which at least one multiple fiber has a representative of twisting $-1$, and the arguments above apply.

Of course, the unique tight contact structure $\xi$ on $\Sigma(2,3,5)$ is well known to have $\tw(\xi) = -1$ (see for example \cite[Theorem 1.3]{Ghiggini:seifert}), so there is no difficulty with this case either.
\end{proof}


Theorem \ref{mainthm} now follows immediately: for $Y = \Sigma(a_1,\ldots, a_n)$ and $\xi$ a contact structure on $Y$, by Theorem \ref{twistthm}, the twisting number of $\xi$ is at least $-\sqrt{a_1\cdots a_n}$. By Corollary \ref{ratboundcor}, $(Y,\xi)$ does not admit a contact type embedding in $(\arr^4, \omega_{std})$. 

\section{Further applications}\label{examplesec} 

\subsection{Planar open books and strong cobordisms to $S^3$}\label{examplesec1}
It is worth noting that, while our obstruction to bounding a symplectically convex domain is the same as the obstruction in \cite{OSS:planar} for a $\xi$ to be supported by a planar open book, we have not proved that no tight structure on a Brieskorn homology sphere is planar. In that direction, we have the following.

\begin{theorem}\label{planarthm} Let $Y$ be a Brieskorn homology sphere with contact structure $\xi$, and let $X_\Gamma$ be the negative definite plumbed 4-manifold with boundary $Y$ as in Section \ref{estimatesec}. Assume that $c^+(\xi)$ is nonzero, and either
\begin{itemize}
\item[(a)] There is a symplectic structure on $X_\Gamma$ that is a filling of $\xi$, or
\item[(b)] The intersection form of $X_{\Gamma}$ is diagonalizable.
\end{itemize}
Then $\xi$ is not supported by a planar open book decomposition. Moreover, if (b) holds, then the contact invariant $c^+(\xi)$ has nonzero image in $HF^+_{red}(-Y)$.
\end{theorem}

\begin{proof} If $(X_\Gamma, \omega)$ is a symplectic filling of $(Y,\xi)$, and the intersection form of $X_\Gamma$ is not diagonalizable, then $\xi$ is not supported by a planar open book by \cite[Theorem 1.2]{Etnyre:planar}. 

Hence, we now assume $X_\Gamma$ has diagonalizable intersection form. By Corollary \ref{taudiffcor} combined with Theorem \ref{twistthm}, we have $\tau_\xi(Y,K) > \tau_{sm}(Y,K)$ for $K$ a regular fiber of the Seifert structure on $Y$. By Lemma \ref{taulemma} the contact invariant $c^+(\xi)\in \ker(U)\subset HF^+(-Y)$ is a nonzero element not equal to $\Theta^+$, the unique nonzero element in the image of $HF^\infty\to HF^+$ and in the kernel of $U$. Hence $c^+(\xi)$ has nonzero image in the quotient $HF^+_{red}(-Y) = HF^+(-Y)/\im(HF^\infty(-Y))$, so $\xi$ is not planar by \cite[Theorem 1.2]{OSS:planar}.
\end{proof}

Observe that the two cases in Theorem \ref{planarthm} are necessary, as seen in the case of $\Sigma(2,3,5)$. This Brieskorn sphere has a unique tight contact structure, which is filled by a Stein structure on the plumbed manifold $X_\Gamma$, where $\Gamma$ is the (non-diagonalizable) negative-definite $E_8$ graph. Moreover, since $HF^+_{red}(-\Sigma(2,3,5)) = 0$, the last line of Theorem \ref{planarthm} does not hold in this case.

On the other hand, the authors are not aware of a fillable contact structure on a Brieskorn homology sphere where neither of the two conditions in the preceding theorem hold, nor yet of a tight structure on a (correctly-oriented) Brieskorn sphere that is not fillable.

\begin{conjecture} No tight, positive contact structure on a Brieskorn homology sphere is supported by a planar open book decomposition.
\end{conjecture}

Combining Theorem \ref{planarthm} with Theorem \ref{obstrthm}, we obtain:

\begin{corollary}\label{strongcobordcor} If $Y$ is a Brieskorn homology sphere such that the negative definite plumbed manifold $X_\Gamma$ bounding $Y$ has diagonalizable intersection form, and $\xi$ is a contact structure on $Y$ with $c^+(\xi)$ nonzero, then there is no strong symplectic cobordism from $(Y,\xi)$ to $S^3$.
\end{corollary}

Note that the contact structure on $S^3$ is immaterial: the cited theorems rule out the standard structure $\xi_0$, while every other contact structure on $S^3$ is overtwisted and hence has vanishing contact invariant. Such contact structures cannot be reached by a strong cobordism from a contact structure with nonzero contact invariant by Echeverria's naturality result. Also note that, as in (b) of Theorem \ref{planarthm}, it is not assumed that $X_\Gamma$ admits any symplectic structure filling $\xi$; in fact $\xi$ is not assumed fillable.

However, as mentioned above, we do not know of an example of a positive tight contact structure on a Brieskorn homology sphere that is not fillable. In the fillable case, the existence of a strong cobordism to $S^3$ is ``nearly'' equivalent to the existence of a contact-type embedding in $(\arr^4, \omega_{std})$ in the sense of the Proposition below, which is probably known to experts.

\begin{proposition}\label{strongcobordprop} Let $Y$ be an oriented 3-manifold and $\xi$ a strongly symplectically fillable contact structure on $Y$. There is a strong cobordism from $(Y,\xi)$ to $(S^3, \xi_0)$ if and only if there exists an integer $k\geq 0$ and a symplectic structure $\omega_k$ on a $k$-fold blowup $X_k \cong \arr^4\#^k {\overline{\cee P}}^2$ of $(\arr^4, \omega_{std})$, equal to $\omega_{std}$ outside a compact set, such that $(Y,\xi)$ embeds as the boundary of a symplectically convex domain in $X_k$.

Moreover, the preceding holds if and only if for every strong symplectic filling $(W, \omega_W)$ of $(Y,\xi)$ there exists $k$ and a symplectic embedding $(W, \omega_W)\to (X_k, \omega_k)$ with $\omega_k$ as before.
\end{proposition}

\begin{proof} Thinking of $X_k$ as the blowup of $\arr^4$ at points $p_1,\ldots, p_k$, the hypothesis on $\omega_k$ means that the blowdown map $X_k \to \arr^4$ is a symplectomorphism when restricted to the complement of the preimage $\widetilde{B}$ of any sufficiently large ball $B\subset \arr^4$. If $(Y,\xi)$ is the boundary of a symplectically convex domain $W\subset X_k$, then for $B$ large enough we have $W\subset \widetilde{B}$, and $\widetilde{B} - \Int(W)$ is a strong cobordism from $(Y,\xi)$ to $(S^3, \xi_0)$.

Conversely, suppose $(W, \omega_W)$ is a strong filling of $(Y,\xi)$ and that $(Z, \omega_Z)$ is a strong cobordism from $(Y,\xi)$ to $(S^3, \xi_0)$. Then a standard argument shows there is a smooth symplectic form on $W\cup_Y Z$ agreeing with the given forms on $W$ and $Z$, and therefore strongly filling $\partial(W\cup_Y Z) = (S^3,\xi_0)$. It is well known by work of McDuff \cite{mcduff90} that every strong symplectic filling of $S^3$ is diffeomorphic to the $k$-fold blowup of $B^4$ for some $k$, hence attaching the positive symplectization of $(S^3, \xi_0)$ to $\partial(W\cup_Y Z)$ yields a symplectic manifold $(X_k,\omega_k)$ as in the statement, containing $(Y,\xi)$ as the boundary of the symplectically convex domain $W$.
\end{proof}

This gives an analog of Theorem \ref{planarthm} for the question of strong cobordisms to $S^3$:

\begin{corollary} Let $Y$ be a Brieskorn integer homology sphere and $\xi$ a contact structure with $c^+(\xi)$ nonzero. Assume that either condition (a) or (b) of Theorem \ref{planarthm} holds. Then there is no strong symplectic cobordism from $(Y,\xi)$ to $S^3$.
\end{corollary}

Indeed, if (a) holds and there is a strong cobordism to $S^3$, then by Proposition \ref{strongcobordprop} it must be that $X_\Gamma$ embeds in a blowup $X_k$ of $\arr^4$. Such a blowup has negative-definite, diagonalizable intersection form, and since $Y$ is an integer homology sphere the intersection form of $X_\Gamma$ is a direct sum factor of that of $X_k$. Hence the intersection form of $X_\Gamma$ is also diagonalizable, which is to say (b) of Theorem \ref{planarthm} holds as well. This contradicts Corollary \ref{strongcobordcor}.


\subsection{Additional examples}\label{examplesec2}
  
Here we provide some examples and discussion relevant to the questions mentioned in the introduction, particularly the implications in \eqref{implications}. Sometimes we will relax the condition that $Y$ be an integer homology sphere. 

As a first question, one can ask if the implication B2$\implies$B1, that a homology sphere that embeds in $\arr^4$ bounds a homology ball, can be reversed. If we consider homology with rational coefficients then the answer is ``no,'' which we can see as follows.  For any rational homology sphere $Y$, the connected sum $Y\#-Y$ bounds a rational homology ball, namely $(Y - B^3)\times I$. But the connected sum need not embed in $\arr^4$: taking $Y = L(p,q)$ to be a lens space, it is a consequence of work of Zeeman \cite{zeeman} and Epstein \cite{epstein} that $L(p,q)-B^3$ admits a smooth embedding in $\arr^4$ if and only if $p$ is odd. Hence $L(2p,q)\#L(2p, 2p-q)$ satisfies B1 but not B2. The authors are unaware of such an example among integer homology spheres.

There are many examples of 3-manifolds (including integer homology spheres) that satisfy B2 but neither B3 nor B3'. Indeed, if $(Y,\xi)$ is a contact homology sphere embedded in $\cee^2$ either as a hypersurface of contact type or as the boundary of a Stein domain, then $\xi$ is strongly symplectically filled by the bounded component $W$ of the complement of $Y$, which is a homology ball. Recall that the homotopy class of $\xi$ as a tangent plane field is captured by the numerical invariant $\theta(\xi)$, which by definition equals $c_1^2(W, J) - 3\sigma(W) - 2\chi(W)$ for any almost-complex 4-manifold $(W,J)$ with boundary $Y$, such that $\xi$ is $J$-invariant. In our case $c_1(W, J)$ is necessarily zero, so $\theta(\xi) = -2$. However, there are many examples of homology spheres that embed smoothly in $\arr^4$ but do not carry any symplectically fillable contact structure $\xi$ having $\theta(\xi) = -2$, such as $-M_p$ for $p\geq 2$, where $M_p$ is the Seifert rational homology sphere considered in Section \ref{unobstrsec} (see \cite[Lemma~$21$]{Tosun20}). Irreducible integer homology spheres not carrying fillable structures with $\theta = -2$ include $-\Sigma(2,3,12n + 1)$ for $n\geq 1$\cite{MT:pseudoconvex}, though for $n\geq 3$ is unknown if these manifolds admit smooth embeddings in $\arr^4$.

In light of Theorem \ref{mainthm}, Gompf's Conjecture \ref{gompfconj} would imply that all Brieskorn spheres fail to satisfy both B3 and B3', for reasons that in many cases must be deeper than these homotopy considerations.

Further examples relating to B3, B3' and B4 will make repeated use of the following result of Gompf.

\begin{theorem}[Gompf \cite{gompfemb}]\label{gompfthm} Let $W$ be a compact 4-manifold with boundary, $(X, J_X)$ a complex surface with complex structure $J_X$, and $\psi: W\to X$ a smooth embedding. If the induced complex structure $\psi^*(J_X)$ is homotopic through almost-complex structures to a Stein structure $J_W$ on $W$, then $\psi$ is smoothly isotopic to a holomorphic embedding $\tilde\psi: (W, J_W)\to (X, J_X)$. In particular, $\tilde\psi(W)$ is a Stein domain in $X$ biholomorphic to $(W, J_W)$.
\end{theorem}

We use the term ``Stein embedding'' for a map satisfying the properties of $\tilde\psi$ in this theorem.

As a special case, note that if $W$ is an integer homology ball, then $W$ carries a unique homotopy class of almost-complex structure. Thus if $W$ admits a Stein structure $J$ (as an abstract manifold), then any smooth embedding of $W$ into a complex manifold is isotopic to a Stein embedding of $(W, J)$.


\subsection{Hyperbolic examples}\label{hypexsec} Here we provide a family of examples of integer homology spheres satisfying B3' but not B3, when the contact structure is fixed. For an integer $n>0$, consider the smooth 4-manifold $W_n$ whose handle description is given in Figure \ref{PCStein}.
\begin{figure}[h!]
\begin{center}
 \includegraphics[width=7cm]{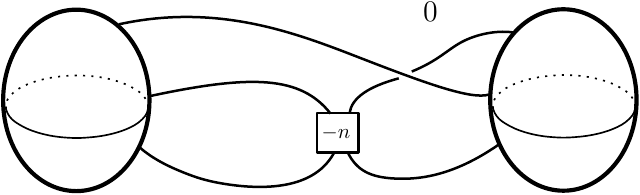}
 \caption{Contractible manifold $W_n$. The box indicates $n$ full left-handed twists.}
  \label{PCStein}
\end{center}
\end{figure} 

For $n= 1$, this is the well-known Mazur cork considered by Akbulut in \cite{akbulut:cork}. In general, it is easy to see that $W_n$ is contractible, and embeds in $\arr^4$ (this holds by the general principle mentioned after Conjecture \ref{kollarconj}, or here by observing that the double of $W_n$ is $S^4$ using handle calculus \cite{GS}). Furthermore, it is straightforward to convert the diagram of Figure \ref{PCStein} to a Stein diagram: one realizes the 0-framed circle as a Legendrian with Thurston--Bennequin number 2, so after a single stabilization we obtain a Stein structure $J_n$ on $W_n$ (see \cite[Figure 3]{KOU}). Write $\xi_n$ for the associated contact structure on $Y_n = \partial W_n$. Strictly there is a choice involved in the stabilization, but the resulting contact structures are contactomorphic; we fix one such choice.

According to \cite[Theorem 1.2(2)]{KOU} (see also \cite{AKcork}), the contact invariant $c^+(\xi_n)$ has nonzero image in $HF^+_{red}(-Y_n)$. By Corollary \ref{reducedcor} this fact obstructs $(Y_n,\xi_n)$ from being a symplectically convex boundary, and using Theorem \ref{gompfthm} we obtain:

\begin{proposition} There exists an infinite family $(Y_n, \xi_n)$ of contact structures on hyperbolic integer homology 3-spheres such that $Y_n$ embeds smoothly in $\cee^2$ as the boundary of a contractible Stein domain $W_n\subset \cee^2$ with induced contact structure $\xi_n$, but $\xi_n$ does not arise from a contact type embedding of $Y_n$ in $(\arr^4, \omega_{std})$. In particular, $W_n$ is not isotopic to a symplectically convex domain by any isotopy that preserves the contact structure on the boundary.
\end{proposition}

The fact that the $Y_n$ are hyperbolic is given in \cite[Theorem 1.2(3)]{KOU}.

In particular, $(W_n, J_n)$ is not Stein equivalent to any rationally convex domain in $\cee^2$. Moreover, recall that a Stein domain $(W,J)$ (not necessarily in $\cee^2$) has a unique corresponding homotopy class of compatible Weinstein structure $(W, \omega_W, v)$ where $\omega_W$ is a symplectic structure and $v$ a ``gradient-like'' Liouville field \cite{CE} (indeed, one can take $\omega_W = -dd^c\phi$ for a defining strictly $J$-convex Morse function $\phi$, and $v$ the gradient of $\phi$ with respect to the K\"ahler metric induced by $J$ and $\omega_W$). Since a Weinstein domain is symplectically convex, we infer:

\begin{corollary} There is no compatible Weinstein structure $(\omega_{W_n}, v)$ on $(W_n, J_n)$ that embeds symplectically in $(\arr^4, \omega_{std})$.
\end{corollary}

Again, these examples indicate that B3' need not imply B3. Consider the converse, whether the condition B3 ($Y$ embeds in $\arr^4$ as a hypersurface of contact type) implies B3' ($Y$ embeds as the boundary of a Stein domain). If the symplectically convex region $W$ bounded by a contact type hypersurface admits a Weinstein structure, then it also admits a homotopic Stein structure, and hence by Theorem \ref{gompfthm} it is isotopic to a Stein domain. But from the above corollary, we see that should such a Weinstein structure exist, it need not be apparent in the ambient symplectic structure---one may need to modify the symplectic structure on the domain. 

To add further interest to this family of examples, Karakurt-Oba-Ukida \cite{KOU} proved that the manifold $W_n$ carries {\it another} Stein structure, and for this structure the situation regarding an isotopy to a symplectically convex domain is not clear. Indeed, those authors construct an ``allowable'' symplectic Lefschetz fibration on $W_n$, having fibers of genus 0. This structure determines a deformation class of Stein structure $J'_n$ on $W_n$ and in particular a contact structure $\xi_n'$ on $Y_n$ supported by a genus 0 open book on the boundary. It follows \cite[Theorem 1.2]{OSS:planar} that $c^+(\xi_n')$ vanishes in $HF^+_{red}(-Y_n)$, so we gain no information on whether $\xi_n'$ can arise as a symplectically convex boundary. 

By Theorem \ref{gompfthm}, the smooth embedding of $W_n$ in $\cee^2$ is isotopic to Stein embeddings with respect to {\it either} $J_n$ or $J_n'$. The images of these embeddings are then smoothly isotopic contractible Stein domains that induce different contact structures on their boundaries; one of these domains cannot be symplectically convex. It seems an interesting problem to determine whether $(W_n, J_n')$ can be made symplectically convex.

\subsection{Unobstructed Seifert examples}\label{unobstrsec} We now examine a family of examples (Seifert rational homology spheres) generalizing the manifold shown by Nemirovski-Siegel not to bound a symplectically convex domain. For an integer $p\geq 2$, consider the Seifert fibered manifold $M_p = M(-1; \frac{p-1}{p}, \frac{1}{p}, \frac{1}{p})$. This manifold has a surgery description similar to that in Figure \ref{small}, with the coefficient 0 replaced by $-1$, and three meridional surgery curves with coefficients $-\frac{p}{p-1}$, $-p$, and $-p$. It is not hard to see that $M_p$ is diffeomorphic to the boundary of the 4-manifold $Z_p$ (a rational homology ball) with handle description given in Figure \ref{NSmfds}. Moreover, $Z_2$ is the disk bundle over $\arr P^2$ with Euler number $-2$, which was considered by Nemirovski-Siegel \cite{NemSie}. 

\begin{proposition} For every integer $p\geq 2$, the manifold $Z_p$ admits a Stein structure $J_p$ and an embedding of $(Z_p, J_p)$ as a Stein domain in $\cee^2$. The symplectic structure corresponding to $J_p$ is compatible with an allowable Lefschetz fibration on $Z_p$ having fibers of genus 0, and the corresponding contact structure on $M_p$ is supported by a planar open book.
\end{proposition}

\begin{figure}[h!]
\begin{center}
 \includegraphics[width=13cm]{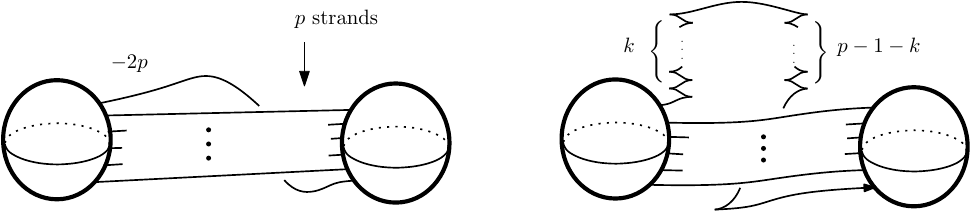}
 \caption{On the left is the 4-manifold $Z_p$, and on the right is the Stein domain $(Z_p, J_{p,k})$.}
  \label{NSmfds}
\end{center}
\end{figure} 

Since the contact structure $\xi_p$ induced by $J_p$ is planar, our obstruction gives no information on whether $Z_p$ can be rationally convex (in fact, as noted below, $HF^+_{red}(-M_p) = 0$). However, for $p = 2$, Nemirovski-Siegel showed that $(M_2, \xi_2)$ is {\it not} the boundary of a rationally convex domain in $\cee^2$, in particular in the terminology of \cite{CE}, $(Z_2, J_2)$ is $i$-convex but not symplectically convex. 

The above implies that in general the condition that a Stein domain in $\cee^2$ admit a genus 0 Lefschetz fibration compatible with its symplectic structure is not sufficient to ensure that the Stein domain is isotopic to a symplectically convex domain. Hence the symplectic convexity of the examples $(W_n, J_n')$ from the previous subsection cannot be decided by a general argument based on planarity.

\begin{proof}
First note that for any $p$, the manifold $M_p$ admits a smooth embedding in $\arr^4$: this was observed by Casson and Harer (part (3) of the main theorem in \cite{CH}). In fact, $Z_p$ also embeds in $\arr^4$. To see this, it is convenient to modify the diagram on the left of Figure \ref{NSmfds} by dragging the undercrossing strand around one attaching ball of the 1-handle, causing it to pass over the remaining strands. This requires the framing coefficient to change from $-2p$ to $0$. Now add a 2-handle along a trivial circle passing over the 1-handle with framing 0. Then we can cancel the 1-handle of $Z_p$ with the new handle, leaving behind a 2-handle attached along a 0-framed unknot. Adding a 3-handle then yields the 4-ball $B^4\subset\arr^4$. So $Z_p$ embeds in $\arr^4$ smoothly.

We wish to apply Theorem \ref{gompfthm} above. For this we observe that almost-complex structures on $Z_p$ are classified up to homotopy by their induced \spinc structure, and the set of \spinc structures is in bijection with $H^2(Z_p; \zee) \cong \zee/p\zee$. We claim that all these homotopy classes of almost-complex structure are realized by Stein structures derived from Figure \ref{NSmfds}. To see this, observe that the $-2p$-framed circle $K$ in $Z_p$ can be realized as a Legendrian knot with Thurston-Bennequin number $-p$ and rotation number $1$. Therefore after $p-1$ stabilizations we obtain a Legendrian representative of $K$ having Thurston-Bennequin number $-2p +1$. Since this is one greater than the smooth framing coefficient, the (unique) Stein structure on the 1-handle extends across the 2-handle \cite{Eliashberg:stein,Gompf}. Now, there are $p$ choices in how to perform the stabilizations: for each $k\in\{0,\ldots, p-1\}$ we can make $k$ negative and $p-1-k$ positive stabilizations as shown on the right of Figure \ref{NSmfds}. It is straightforward to check, using the methods of \cite{Gompf} or \cite[Section 11.3]{GS} (see particularly Example 11.3.12 of the latter), that the Stein structures $J_{p,k}$ arising from different choices of $k$ are not homotopic: indeed, they induce non-homotopic contact structures on the boundary $M_p$. Thus as $k$ varies, the Stein structures $J_{p,k}$ give representatives of each homotopy class of almost-complex structure on $Z_p$. It follows that if $\varphi: Z_p\to\cee^2$ is the embedding constructed above, the induced complex structure $\varphi^*(J_{std})$ is homotopic to some Stein structure $J_{p,k}$, and by Gompf's result $\varphi$ is isotopic to a Stein embedding of this $(W_p, J_{p,k})$. Below we simply write $J_p$ for this distinguished Stein structure.

Recall that a rational homology sphere $Y$ is an $L$-{\it space} if $HF^+_{red}(Y) = 0$, and $Y$ is an $L$-space if and only if $-Y$ is. According to work of Lisca-Stipsicz \cite[Theorem 1.1]{LS:hf3}, a small Seifert manifold such as $M_p$ is an $L$-space exactly when either $M_p$ or $-M_p$ fails to admit a contact structure transverse to the Seifert structure. Furthermore, a criterion of Lisca-Mati\'c \cite{LMfoliations} shows that the Seifert manifold $M(-1; r_1, r_2, r_3)$ admits a transverse contact structure if and only if there exist relatively prime integers $a$ and $m$, with $0<a<m$, such that
\[
mr_1 < a < m(1-r_2) \quad \mbox{and} \quad mr_3 < 1,
\]
where we have arranged $r_1\geq r_2\geq r_3$. It is easy to see that this condition does not hold for $M(-1; \frac{p-1}{p}, \frac{1}{p}, \frac{1}{p})$, so that the latter is an $L$-space. It now follows from \cite[Corollary 1.7]{LS:hf3} that {\it every} contact structure on $M_p$ is planar, in particular the the contact structure $\xi_p$ induced by $(Z_p, J_p)$. In \cite[Theorem~$1$]{Wendl:planar}, Wendl shows that a planar open book can always be extended to an allowable Lefschetz fibration over any minimal symplectic filling. Hence, $Z_p$ admits an allowable Lefschetz fibration having fibers of genus 0. 
\end{proof}

\begin{remark} Clearly the induced complex structure $\varphi^*(J_{std})$ has vanishing first Chern class. When $p$ is odd, the homotopy class of $J_{p,k}$ is uniquely determined by the Chern class, and in particular the  Stein structure arising when $k = 0$ is the only one with trivial Chern class (see \cite[Proposition 2.3]{Gompf}). When $p$ is even there is a second such Stein structure, namely the one with $k = \frac{p}{2}$.
\end{remark}

One can explicitly construct planar open books for (all) contact structures on $M_p$, using the fact that any such contact structure is given by a contact surgery diagram of the form in \cite[Figure 2]{LS:hf3}. Such techniques are used, for example, in \cite{Schoenenberger05,PVHM}, and indeed give the proof of \cite[Corollary 1.7]{LS:hf3} cited above.

Finally, we point out a generalization of the manifolds $M_p$ studied above. In \cite[Theorem~$1.1$]{IM}, Issa-McCoy found a two parameter family of rational homology spheres 
\[
M_{p,\ell}=M(-\ell; \frac{1}{p}, \frac{p-1}{p}, \frac{1}{p}, \cdots, \frac{p-1}{p}, \frac{1}{p}),
\]
 having $2\ell+1$ singular fibers, and proved that for every $p\geq 2, \ell\geq 1$, $M_{p,\ell}$ embeds smoothly in $\arr^4$. In this notation, $M_p$ above is $M_{p,1}$. 

With the single exception of $p =2$, $\ell =1$, it remains an open question whether $M_{p,\ell}$ embeds as a hypersurface of contact type in $\arr^4$.


\bibliography{references}

\providecommand{\bysame}{\leavevmode\hbox to3em{\hrulefill}\thinspace}
\providecommand{\MR}{\relax\ifhmode\unskip\space\fi MR }
\providecommand{\MRhref}[2]{%
  \href{http://www.ams.org/mathscinet-getitem?mr=#1}{#2}
}
\providecommand{\href}[2]{#2}
\begin{thebibliography}{10}

\bibitem{akbulut:cork}
Selman Akbulut, \emph{A fake compact contractible 4-manifold}, Journal of
  Differential Geometry \textbf{33} (1991), no.~2, 335--356.

\bibitem{AKcork}
Selman Akbulut and {\c{C}a\u{g}r\i}~Karakurt, \emph{Action of the cork twist on
  {F}loer homology}, Proceedings of the {G}\"{o}kova {G}eometry-{T}opology
  {C}onference 2011, Int. Press, Somerville, MA, 2012, pp.~42--52. \MR{3076042}

\bibitem{CH}
Andrew Casson and John Harer, \emph{Some homology lens spaces which bound
  rational homology balls}, Pacific Journal of Mathematics \textbf{96} (1981),
  no.~1, 23--36.

\bibitem{weimin}
Weimin Chen, \emph{Contact splitting of symplectic {$\Bbb Q$}-homology
  {$\Bbb{CP}^2$}}, Proceedings of the {G}\"{o}kova {G}eometry-{T}opology
  {C}onference 2017, Int. Press, Somerville, MA, 2018, pp.~53--72. \MR{3838085}

\bibitem{CE}
K.~Cieliebak and Y.~Eliashberg, \emph{From {S}tein to {W}einstein and {B}ack --
  {S}ymplectic {G}eometry of {A}ffine {C}omplex {M}anifolds}, Colloquium
  Publications, vol.~59, Amer. Math. Soc., 2012.

\bibitem{CE-qconvex}
Kai Cieliebak and Yakov Eliashberg, \emph{The topology of rationally and
  polynomially convex domains}, Invent. Math. \textbf{199} (2015), no.~1,
  215--238. \MR{3294960}

\bibitem{CGHsummary}
Vincent Colin, Paolo Ghiggini, and Ko~Honda, \emph{{$HF = ECH$ via open book
  decompositions: a summary}}, arXiv:1103.1290.

\bibitem{CGH3}
Vincent Colin, Paolo Ghiggini, and Ko~Honda, \emph{{The equivalence of Heegaard
  Floer homology and embedded contact homology III: from hat to plus}},
  arXiv:1208.1526.

\bibitem{CGHall}
Vincent Colin, Ko~Honda, and Paolo Ghiggini, \emph{{The equivalence of Heegaard
  Floer homology and embedded contact homology}}, arxiv:1208.1074,
  arXiv:1208.1077, arXiv:1208.1526.

\bibitem{DLVVW}
Aliakbar Daemi, Tye Lidman, David~Shea Vela-Vick, and C.\ M.~Michael Wong,
  \emph{Ribbon homology cobordisms}, arXiv:1904.09721.

\bibitem{donaldson83}
S.~K. Donaldson, \emph{An application of gauge theory to four-dimensional
  topology}, J. Differential Geom. \textbf{18} (1983), no.~2, 279--315.
  \MR{710056}

\bibitem{DS}
Julien Duval and Nessim Sibony, \emph{Polynomial convexity, rational convexity,
  and currents}, Duke Math. J. \textbf{79} (1995), no.~2, 487--513.
  \MR{1344768}

\bibitem{mariano}
Mariano Echeverria, \emph{{Naturality of the contact invariant in monopole
  Floer homology under strong symplectic cobordisms}}, Algebr. Geom. Topol.
  \textbf{20} (2020), 1795--1875.

\bibitem{E-filling}
Yakov Eliashberg, \emph{Filling by holomorphic discs and its applications},
  Geometry of low-dimensional manifolds, 2 ({D}urham, 1989), London Math. Soc.
  Lecture Note Ser., vol. 151, Cambridge Univ. Press, Cambridge, 1990,
  pp.~45--67. \MR{1171908}

\bibitem{Eliashberg:stein}
Yakov Eliashberg, \emph{Topological {C}haracterization of {S}tein {M}anifolds
  of {D}imension {$>2$}}, Internat. J. Math. \textbf{1} (1990), no.~1, 29--46.
  \MR{1044658 (91k:32012)}

\bibitem{epstein}
D.~B.~A. Epstein, \emph{Embedding punctured manifolds}, Proc. Amer. Math. Soc.
  \textbf{16} (1965), 175--176. \MR{208606}

\bibitem{Etnyre:planar}
John~B. Etnyre, \emph{Planar {O}pen {B}ook {D}ecompositions and {C}ontact
  {S}tructures}, IMRN \textbf{79} (2004), 4255--4267.

\bibitem{EH:nonexistence}
John~B. Etnyre and Ko~Honda, \emph{On the nonexistence of tight contact
  structures}, Ann. of Math. (2) \textbf{153} (2001), no.~3, 749--766.
  \MR{1836287}

\bibitem{FS87}
Ronald Fintushel and Ronald~J. Stern, \emph{{${\rm O}(2)$} actions on the
  5-sphere}, Invent. Math. \textbf{87} (1987), no.~3, 457--476. \MR{874031}

\bibitem{FS:seifert}
\bysame, \emph{Instanton homology of {S}eifert fibred homology three spheres},
  Proc. London Math. Soc. (3) \textbf{61} (1990), no.~1, 109--137. \MR{1051101}

\bibitem{furuta90}
Mikio Furuta, \emph{Homology cobordism group of homology {$3$}-spheres},
  Invent. Math. \textbf{100} (1990), no.~2, 339--355. \MR{1047138}

\bibitem{geiges95}
{Hansj\"{o}rg} Geiges, \emph{Examples of symplectic {$4$}-manifolds with
  disconnected boundary of contact type}, Bull. London Math. Soc. \textbf{27}
  (1995), no.~3, 278--280. \MR{1328705}

\bibitem{Ghiggini:fillability}
Paolo Ghiggini, \emph{Ozsv{\'a}th-{S}zab{\'o} invariants and fillability of
  contact structures}, Math. Z. \textbf{253} (2006), no.~1, 159--175.
  \MR{2206641}

\bibitem{Ghiggini:seifert}
\bysame, \emph{On tight contact structures with negative maximal twisting
  number on small {S}eifert manifolds}, Algebr. Geom. Topol. \textbf{8} (2008),
  no.~1, 381--396. \MR{2443233}

\bibitem{GS:classification}
Paolo Ghiggini and Stephan Sch{\"{o}}nenberger, \emph{On the classification of
  tight contact structures}, Topology and geometry of manifolds ({A}thens,
  {GA}, 2001), Proc. Sympos. Pure Math., vol.~71, Amer. Math. Soc., Providence,
  RI, 2003, pp.~121--151. \MR{2024633}

\bibitem{Gompf}
Robert~E. Gompf, \emph{Handlebody {C}onstruction of {S}tein {S}urfaces}, Ann.
  Math. \textbf{148} (1998), 619--693.

\bibitem{gompfemb}
\bysame, \emph{Smooth embeddings with {S}tein surface images}, J. Topol.
  \textbf{6} (2013), no.~4, 915--944. \MR{3145144}

\bibitem{GS}
Robert~E. Gompf and Andr\'{a}s~I. Stipsicz, \emph{{$4$}-manifolds and {K}irby
  calculus}, Graduate Studies in Mathematics, vol.~20, American Mathematical
  Society, Providence, RI, 1999. \MR{1707327}

\bibitem{gordon81}
C.~McA. Gordon, \emph{Ribbon concordance of knots in the {$3$}-sphere}, Math.
  Ann. \textbf{257} (1981), no.~2, 157--170. \MR{634459}

\bibitem{gromov85}
M.~Gromov, \emph{Pseudo holomorphic curves in symplectic manifolds}, Invent.
  Math. \textbf{82} (1985), no.~2, 307--347. \MR{809718}

\bibitem{heddentau}
Matthew Hedden, \emph{An {O}zsv\'{a}th-{S}zab\'{o} {F}loer homology invariant
  of knots in a contact manifold}, Adv. Math. \textbf{219} (2008), no.~1,
  89--117. \MR{2435421}

\bibitem{Honda:classification1}
Ko~Honda, \emph{On the {C}lassification of {T}ight {C}ontact {S}tructures {I}},
  Geom. Topol. \textbf{4} (2000), 309--368.

\bibitem{IM}
Ahmad Issa and Duncan McCoy, \emph{{Smoothly embedding Seifert fibered spaces
  in $S^4$}}, Preprint, arxiv:1810.04770, 2018.

\bibitem{JTZ}
Andr{\'{a}}s Juh{\'{a}}sz, Dylan Thurston, and Ian Zemke, \emph{{Naturality and
  mapping class groups in Heegaard Floer homology}}, To appear in Mem. AMS.

\bibitem{KOU}
{\c{C}a\u{g}r\i}~Karakurt, Takahiro Oba, and Takuya Ukida, \emph{Planar
  {L}efschetz fibrations and {S}tein structures with distinct
  {O}zsv\'{a}th-{S}zab\'{o} invariants on corks}, Topology Appl. \textbf{221}
  (2017), 630--637. \MR{3624490}

\bibitem{kollar}
J\'{a}nos Koll\'{a}r, \emph{Is there a topological {B}ogomolov-{M}iyaoka-{Y}au
  inequality?}, Pure Appl. Math. Q. \textbf{4} (2008), no.~2, Special Issue: In
  honor of Fedor Bogomolov. Part 1, 203--236. \MR{2400877}

\bibitem{KMbook}
Peter Kroheimer and Tomasz Mrowka, \emph{Monopoles and three-manifolds}, New
  Mathematical Monographs, vol.~10, Cambridge University Press, Cambridge,
  2007.

\bibitem{KLTall}
{\c{C}a\u{g}atay}~Kutluhan, Clifford Taubes, and Yi-Jen Lee, \emph{{Heegaard
  Floer homology and Seiberg-Witten Floer homology}}, arxiv:1007.1979,
  arxiv:1008.1595, arxiv:1010.3456, arxiv:1107.2297, arxiv:1204.0115.

\bibitem{LMfoliations}
Paolo Lisca and Gordana {Mati\'{c}}, \emph{Transverse contact structures on
  {S}eifert 3-manifolds}, Algebr. Geom. Topol. \textbf{4} (2004), 1125--1144.
  \MR{2113899}

\bibitem{LS:hf3}
Paolo Lisca and Andr{\'a}s~I. Stipsicz, \emph{Ozsv{{\'a}}th-{S}zab{{\'o}}
  {I}nvariants and {T}ight {C}ontact {T}hree-{M}anifolds, {III}}, J. Sympl.
  Geom. \textbf{5} (2007), no.~4, 357--384.

\bibitem{MT:pseudoconvex}
Thomas~E. Mark and B\"{u}lent Tosun, \emph{Obstructing pseudoconvex embeddings
  and contractible {S}tein fillings for {B}rieskorn spheres}, Adv. Math.
  \textbf{335} (2018), 878--895. \MR{3836681}

\bibitem{mazur}
Barry Mazur, \emph{A note on some contractible {$4$}-manifolds}, Ann. of Math.
  (2) \textbf{73} (1961), 221--228. \MR{125574}

\bibitem{mcduff90}
Dusa McDuff, \emph{The structure of rational and ruled symplectic
  {$4$}-manifolds}, J. Amer. Math. Soc. \textbf{3} (1990), no.~3, 679--712.
  \MR{1049697}

\bibitem{mcduff91}
\bysame, \emph{Symplectic manifolds with contact type boundaries}, Invent.
  Math. \textbf{103} (1991), no.~3, 651--671. \MR{1091622}

\bibitem{milnor:brieskorn}
John Milnor, \emph{On the {$3$}-dimensional {B}rieskorn manifolds
  {$M(p,q,r)$}}, Knots, groups, and 3-manifolds ({P}apers dedicated to the
  memory of {R}. {H}. {F}ox), 1975, pp.~175--225. Ann. of Math. Studies, No.
  84. \MR{0418127}

\bibitem{MR06}
Tomasz Mrowka and Yann Rollin, \emph{Legendrian knots and monopoles}, Algebr.
  Geom. Topol. \textbf{6} (2006), 1--69. \MR{2199446}

\bibitem{NemSie}
Stefan Nemirovski and Kyler Siegel, \emph{Rationally convex domains and
  singular {L}agrangian surfaces in {$\mathbb{C}^2$}}, Invent. Math.
  \textbf{203} (2016), no.~1, 333--358. \MR{3437874}

\bibitem{Nem}
S.~Yu. Nemirovski\u{\i}, \emph{Finite unions of balls in {$\Bbb C^n$} are
  rationally convex}, Uspekhi Mat. Nauk \textbf{63} (2008), no.~2(380),
  157--158. \MR{2640558}

\bibitem{neumann77}
Walter~D. Neumann, \emph{Brieskorn complete intersections and automorphic
  forms}, Invent. Math. \textbf{42} (1977), 285--293. \MR{463493}

\bibitem{neumannraymond}
Walter~D. Neumann and Frank Raymond, \emph{Seifert manifolds, plumbing, {$\mu
  $}-invariant and orientation reversing maps}, Algebraic and geometric
  topology ({P}roc. {S}ympos., {U}niv. {C}alifornia, {S}anta {B}arbara,
  {C}alif., 1977), Lecture Notes in Math., vol. 664, Springer, Berlin, 1978,
  pp.~163--196. \MR{518415}

\bibitem{oka}
Kiyoshi Oka, \emph{Sur les fonctions analytiques de plusieurs variables. {IX}.
  {D}omaines finis sans point critique int\'{e}rieur}, Jpn. J. Math.
  \textbf{23} (1953), 97--155 (1954). \MR{71089}

\bibitem{OSS:planar}
Peter Ozsv{\'a}th, Andr{\'a}s~I. Stipsicz, and Zolt{\'a}n Szab{\'o},
  \emph{{Planar Open Books and Floer Homology}}, Internat. Math. Res. Notices
  \textbf{2005} (2005), 3385--3401.

\bibitem{OS:hfkgenus}
Peter Ozsv{\'a}th and Zolt{\'a}n Szab{\'o}, \emph{Knot floer homology and the
  four-ball genus}, Geom. Topol. 7(2003) 615-639.

\bibitem{OS:grading}
\bysame, \emph{Absolutely graded {F}loer homologies and intersection forms for
  four-manifolds with boundary}, Adv. Math. \textbf{173} (2003), no.~2,
  179--261. \MR{1957829}

\bibitem{OS:plumbing}
\bysame, \emph{{On the Floer Homology of Plumbed Three-Manifolds}}, Geom.
  Topol. \textbf{7} (2003), no.~1, 185--224.

\bibitem{OS:hfk}
\bysame, \emph{{Holomorphic Disks and Knot Invariants}}, Adv. Math. \textbf{1}
  (2004), 58--116.

\bibitem{OS:hf2}
\bysame, \emph{Holomorphic {D}isks and {T}hree-{M}anifold {I}nvariants:
  {P}roperties and {A}pplications}, Ann. of Math. (2) \textbf{159} (2004),
  no.~3, 1159--1245.

\bibitem{OS:hf1}
\bysame, \emph{Holomorphic {D}isks and {T}opological {I}nvariants for {C}losed
  {T}hree-{M}anifolds}, Ann. of Math. (2) \textbf{159} (2004), no.~3,
  1027--1158.

\bibitem{OS:contact}
\bysame, \emph{Heegaard {F}loer {H}omologies and {C}ontact {S}tructures}, Duke
  Math. J. \textbf{129} (2005), no.~1, 39--61.

\bibitem{plam04}
Olga Plamenevskaya, \emph{Bounds for the {T}hurston-{B}ennequin number from
  {F}loer homology}, Algebr. Geom. Topol. \textbf{4} (2004), 399--406.
  \MR{2077671}

\bibitem{PVHM}
Olga Plamenevskaya and Jeremy Van Horn-Morris, \emph{Planar open books,
  monodromy factorizations and symplectic fillings}, Geom. Topol. \textbf{14}
  (2010), no.~4, 2077--2101. \MR{2740642}

\bibitem{raoux}
Katherine Raoux, \emph{{$\tau$}--invariants for knots in rational homology
  spheres}, Algebr. Geom. Topol. \textbf{20} (2020), no.~4, 1601--1640.
  \MR{4127080}

\bibitem{Rasmussen}
Jacob Rasmussen, \emph{Floer {H}omology and {K}not {C}omplements}, Ph.D.
  thesis, Harvard University, 2003.

\bibitem{savelievbook}
Nikolai Saveliev, \emph{Invariants for homology {$3$}-spheres}, Encyclopaedia
  of Mathematical Sciences, vol. 140, Springer-Verlag, Berlin, 2002,
  Low-Dimensional Topology, I. \MR{1941324}

\bibitem{Schoenenberger05}
Stephan Sch\"onenberger, \emph{Planar open books and symplectic fillings},
  Ph.D. thesis, University of Pennsylvania, 2005.

\bibitem{TaubesECH1}
Clifford~Henry Taubes, \emph{Embedded contact homology and {S}eiberg-{W}itten
  {F}loer cohomology {I}}, Geom. Topol. \textbf{14} (2010), no.~5, 2497--2581.
  \MR{2746723}

\bibitem{Tosun20}
B\"{u}lent Tosun, \emph{Stein domains in $\mathbb{C}^2$ with prescribed
  boundary}, Preprint 2020.

\bibitem{viterbo87}
Claude Viterbo, \emph{A proof of {W}einstein's conjecture in {${\bf R}^{2n}$}},
  Ann. Inst. H. Poincar\'{e} Anal. Non Lin\'{e}aire \textbf{4} (1987), no.~4,
  337--356. \MR{917741}

\bibitem{weinstein78}
Alan Weinstein, \emph{On the hypotheses of {R}abinowitz' periodic orbit
  theorems}, J. Differential Equations \textbf{33} (1979), no.~3, 353--358.
  \MR{543704}

\bibitem{Wendl:planar}
Chris Wendl, \emph{Strongly fillable contact manifolds and {$J$}-holomorphic
  foliations}, Duke Math. J. \textbf{151} (2010), no.~3, 337--384. \MR{2605865}

\bibitem{zeeman}
E.~C. Zeeman, \emph{Twisting spun knots}, Trans. Amer. Math. Soc. \textbf{115}
  (1965), 471--495. \MR{195085}

\end{thebibliography}
\bibliographystyle{amsplain}

\end{document}